\documentclass[onefignum,onetabnum]{siamart190516}
\usepackage[T1]{fontenc}
\usepackage[english]{babel}        % Language awareness
\usepackage[utf8]{inputenc}
\usepackage[stretch=25]{microtype} 
\usepackage{enumerate}
\usepackage{hyperref}
\usepackage{url}
\usepackage{graphicx}
\usepackage{xcolor}
\usepackage{pgfplots}
\usepackage{multirow}
\usepackage{subcaption}
\usepackage{amsmath, amsfonts, amssymb}
\usepackage{algorithm}
\usepackage{algorithmic}

\DeclareMathAlphabet\mathbfcal{OMS}{cmsy}{b}{n}

\newcommand{\RR}{\mathbb R}

\newcommand{\xCC}[1]{\ifmmode \mathcal{#1} \else $\mathcal{#1}$ \fi}

\newcommand{\cH}{\xCC H}

\DeclareMathOperator{\dive}{div}
\DeclareMathOperator{\lap}{\Delta}

\def\ds{\displaystyle}
\def\rent{{\bf ]}\hskip -1.5pt {\bf ]}}
\def\lent{{\bf [}\hskip -1.5pt {\bf [}}

\def\bvarphi{\boldsymbol{\varphi}}

\def\p{\partial}
\def\R{\mathbb{R}}
\def\O{\Omega}

\def\d{{\rm d}}
\def\grad{\nabla}
\def\div{\dive}
\def\Aa{\mathcal{A}}

\def\ov#1{\overline{#1}}
\def\wh#1{\widehat{#1}}

\def\Tt{\mathcal{T}}

\def\bHh{\mathbfcal{H}}

\def\Ee{\mathcal{E}}

\def\sig{\sigma}

\def\dt{{\Delta t}}
\def\bc{{\boldsymbol c}}
\def\bF{{\boldsymbol f}}
\def\bg{{\boldsymbol g}}
\def\bu{{\boldsymbol u}}
\def\bU{{\boldsymbol U}}

\def\bv{{\boldsymbol v}}

\def\bdt{{\boldsymbol{\Delta t}}}

\def\be{\begin{equation}}
\def\ee{\end{equation}}
\def\rent{{\bf ]}\hskip -1.5pt {\bf ]}}
\def\lent{{\bf [}\hskip -1.5pt {\bf [}}

\theoremstyle{plain}
\newtheorem{prop}{Proposition}[section]

\theoremstyle{remark}
\newtheorem{remark}{Remark}[section]

\def\wt#1{{\widetilde{#1}}}

\newcommand{\logLogSlopeTriangle}[5]
{
	\pgfplotsextra
	{
		\pgfkeysgetvalue{/pgfplots/xmin}{\xmin}
		\pgfkeysgetvalue{/pgfplots/xmax}{\xmax}
		\pgfkeysgetvalue{/pgfplots/ymin}{\ymin}
		\pgfkeysgetvalue{/pgfplots/ymax}{\ymax}
		
		\pgfmathsetmacro{\xArel}{#1}
		\pgfmathsetmacro{\yArel}{#3}
		\pgfmathsetmacro{\xBrel}{#1-#2}
		\pgfmathsetmacro{\yBrel}{\yArel}
		\pgfmathsetmacro{\xCrel}{\xArel}
		\pgfmathsetmacro{\lnxB}{\xmin*(1-(#1-#2))+\xmax*(#1-#2)} % in [xmin,xmax].
		\pgfmathsetmacro{\lnxA}{\xmin*(1-#1)+\xmax*#1} % in [xmin,xmax].
		\pgfmathsetmacro{\lnyA}{\ymin*(1-#3)+\ymax*#3} % in [ymin,ymax].
		\pgfmathsetmacro{\lnyC}{\lnyA-#4*(\lnxA-\lnxB)}
		\pgfmathsetmacro{\yCrel}{\lnyC-\ymin)/(\ymax-\ymin)}
		
		\coordinate (A) at (rel axis cs:\xArel,\yArel);
		\coordinate (B) at (rel axis cs:\xBrel,\yBrel);
		\coordinate (C) at (rel axis cs:\xCrel,\yCrel);
		
		\draw[#5]   (A)-- node[pos=0.5,anchor=north] {\scriptsize{1}}
		(B)--
		(C)-- node[pos=0.,anchor=east] {\scriptsize{#4}} %% node[pos=0.5,anchor=west] {#4}
		(A);
	}
}

\begin{document}

\title{A convergent entropy diminishing finite volume scheme for a cross-diffusion system}
\author{
Cl\'ement Canc\`es\thanks{Inria, Université de Lille, CNRS, UMR 8524 - Laboratoire Paul Painlev\'e, F-59000 Lille (\href{mailto:clement.cances@inria.fr}{\tt clement.cances@inria.fr})}
\and
Benoît Gaudeul\thanks{Université de Lille,  CNRS, UMR 8524, Inria - Laboratoire Paul Painlev\'e, F-59000 Lille (\href{mailto:benoit.gaudeul@univ-lille.fr}{\tt benoit.gaudeul@univ-lille.fr})}
}
\maketitle
\begin{abstract}
We study a two-point flux approximation finite volume scheme 
for a cross-diffusion system. The scheme is shown to preserve the key properties of 
the continuous systems, among which the decay of the entropy. 
The convergence of the scheme is established thanks 
to compactness properties based on the discrete entropy - entropy dissipation estimate. 
Numerical results illustrate the behavior of our scheme.
\end{abstract}


\section{Introduction}
\subsection{The system under study}

The system studied in this paper has been originally introduced by \cite{bakhtaCrossdiffusionSystemsNonzero2018b} to model the production of solar panels using vapor deposition. 
In this system, we study the diffusion of $N$ species whose respective concentrations are $U = (u_1,\dots,u_N)$ in a (nonempty) connected bounded open domain $\Omega$ of $\mathbb R^d$ for a fixed time $T$. We denote by $Q_T = (0,T) \times \Omega$. The diffusion occurs through exchanges between different species which are quantified by the matrix $A=(a_{i,j})$ of cross-diffusion coefficients.
It leads to the following system of partial differential equations:
\begin{equation}\label{eq:crossdiff}
\partial_t u_i -\dive\left(\sum_{j=1}^N a_{i,j}\left(u_j\nabla u_i-u_i\nabla u_j\right)\right)=0 \quad \text{in}\; Q_T \; \text{for $i\in\lent 1, N \rent$}.
\end{equation}
The matrix $A$ is assumed to be symmetric with nonnegative coefficients, i.e. $a_{i,j} =a_{j,i}\geq 0$. 
$A$ does not depend on $U$ and thus differs from the diffusion matrix $D(U)=(d_{i,j}(U))$ defined by
\[d_{i,j}(U)=\delta_{i,j}\sum_{k\neq i}a_{i,k}u_k -a_{i,j}u_i,\]
where $\delta_{i,j}$ stands for Kronecker symbol, such that the problem \eqref{eq:crossdiff} rewrites 
\be\label{eq:crossdiff2}
\partial_t U-\dive\left(D(U)\nabla U\right)=0.\ee
System~\eqref{eq:crossdiff2} enters the family of the nonlinear cross-diffusion systems since $D$ depends on $U$ 
and has nonzero off-diagonal entries. 
Challenges both from the analytical and numerical points of view come
from the presence of off-diagonal zeros in $A$. In the previous contributions 
\cite{burgerNonlinearPoissonNernst2012, GJ18, Anita2019}, 
the zeros are integrated through the assumption that the cross-diffusion occurs with 
and only with a solvent specie. Until Section \ref{sec:convergence} we will not make 
any assumption about the zeros of $A$. A non-degeneracy assumption will be further 
assumed {in} Section~\ref{sec:convergence}, but our convergence result could extend 
to the particular cross-diffusion matrices considered in~\cite{GJ18, Anita2019, GJ_arxiv}.

We supplement system \eqref{eq:crossdiff} with no-flux boundary conditions 
\be\label{eq:no-flux}
\sum_{j=1}^N a_{i,j}\left(u_j\nabla u_i-u_i\nabla u_j\right)\cdot n=0\quad \text{on}\; (0,T) \times \partial\Omega, \quad i \in \lent 1, N \rent.
\ee 
The initial concentration $U^0=(u_1^0, \dots, u_N^0)$ is supposed to be measurable and to map $\Omega$ into 
\[
\Aa=\left\lbrace U=(u_1,\dots,u_N)\in\RR_+^N\;\middle \vert\; \sum_{i=1}^N u_i=1\right\rbrace,
\] 
so we write in the condensed form $U^0\in L^\infty(\Omega;\Aa)$, {which means that $U^0$ is measurable 
and takes its values in the bounded subset $\Aa$ of $\R^{N}$}. Finally, we assume that all the chemical 
species under consideration are present:
\be\label{eq:mass_pos}
\int_\Omega u_i^0 \d x >0, \quad \forall i \in \lent1,N\rent. 
\ee

\subsection{Formal structure}
This system has several structural properties, the goal of this subsection is to exhibit them.
The calculations presented in this section are formal: we assume that the solutions to~\eqref{eq:crossdiff} 
enjoy enough regularity to justify the calculations below. 
Rigorous proofs at the continuous level for the system under consideration here  can be found in \cite{bakhtaCrossdiffusionSystemsNonzero2018b,berendsenStrongSolutionsWeakstrong2019} (see also \cite{GJ18}). 
The properties listed here can also be obtained by passing to the limit in the numerical scheme. 
The first property we point out is the conservation of mass for all the species involved in System~\eqref{eq:crossdiff}.
\begin{lemma}[conservation of mass]\label{lem:masscons}
 \eqref{eq:crossdiff} and~\eqref{eq:no-flux} corresponding to an initial data $U^0 \in L^\infty(\O;\Aa)$, then
\[
\int_\Omega u_i(t,x)\d x=\int_\Omega u_i^0(x) \d x, \quad \forall t \in [0,T], \; \forall {i} \in \lent 1, N\rent.
\]
\end{lemma}
\begin{proof}
Let $U$ be a solution of \eqref{eq:crossdiff}, $t\in [0,T], i \in \lent 1, N\rent$, and let $\varphi(x,s)=1_{[0,t]}(s)$. 
With this particular choice of $\varphi$, we have for all $s$ that
\begin{multline*}
\int_{\Omega}\dive\left(\sum_{j=1}^N a_{i,j}\left(u_j\nabla u_i-u_i\nabla u_j\right)\right) \varphi(x,s)\d x\\
=-\int_{\Omega}\sum_{j=1}^N a_{i,j}\left(u_j\nabla u_i-u_i\nabla u_j\right) \nabla\varphi(x,s)\d x=0.
\end{multline*}
Hence, using $\varphi$ as a test function in \eqref{eq:crossdiff}, we have:
\[
\int_0^t \dfrac{d}{ds}\left(\int_\Omega u_i( x,s)\d x \right)\d s=0.
\]
The fundamental theorem of calculus yields the desired lemma.
\end{proof}

The symmetry of the matrix $A=(a_{i,j})$ yields:
\[
\sum_{i=1}^N\sum_{j=1}^N a_{i,j}\left(u_j\nabla u_i-u_i\nabla u_j\right)=0.
\]
Therefore, a solution $U$ to \eqref{eq:crossdiff} satisfies
$
\partial_t \sum_{i=1}^N u_i =0.
$
Admit that $u_i(t,x) \geq 0$ for all $t>0$ (this will be proved in the discrete setting and is proved in 
\cite[Proposition 2.2]{berendsenStrongSolutionsWeakstrong2019} in the continuous setting), then 
the admissibility condition encoded in $\Aa$ is preserved along time. 
\begin{lemma}\label{lem:admsum}
Let $U$ be a solution to \eqref{eq:crossdiff} and~\eqref{eq:no-flux} corresponding to an initial data $U^0 \in L^\infty(\O;\Aa)$, 
then $U(t,x) \in \Aa$ for all $(t,x) \in \Aa$, i.e., $U \in L^\infty(Q_T;\Aa)$.
\end{lemma}

{The system can be derived by passing to the macroscopic limit from a random jump process in the 
spirit of~\cite{GL97,BDFPS10}. As expected because of this derivation from statistical physic considerations, the system 
fulfills Onsager's reciprocal relation~\cite{Ons31-I, Ons31-II} and has a formal gradient flow structure. The driving functional is
the mixing entropy
\be\label{eq:ent.mix}
E: \begin{cases}
L^\infty(\O;\Aa) \to \R, \\
U \mapsto \int_\Omega \sum_{i=1}^N u_i\ln(u_i) \d x.
\end{cases}
\ee}
The next property we want to highlight at the continuous level is the decay of this entropy. 
Using the chain rule $\nabla c=c\nabla \ln(c)$, the system \eqref{eq:crossdiff} {is formally equivalent to}
\begin{equation}\label{eq:entropyform}
\partial_t u_i -\dive\left(\sum_{j=1}^N a_{i,j}u_i u_j \left(\nabla \ln(u_i)-\nabla \ln(u_j)\right)\right)=0, 
\quad i \in \lent 1 , N \rent.
\end{equation}
\begin{proposition} \label{prop:entropy}
$E$ is a Lyapunov functional for the 
system~\eqref{eq:no-flux}--\eqref{eq:entropyform}. 
More precisely, the following entropy - entropy dissipation estimate holds:
\begin{equation}\label{eq:EDE}
\dfrac{d}{dt}E(U) +\int_\Omega \left(\sum_{1\leq i<j\leq N}a_{i,j}u_i u_j\left|\nabla \ln(u_i)-\nabla \ln u_j\right|^2\right)\d x=0.
\end{equation}
\end{proposition}
\begin{proof} 
First, we notice that thanks to the conservation of mass:
\[
	\dfrac{d}{dt}E(U)=\dfrac{d}{dt}\int_\Omega \sum_{i=1}^N u_i(\ln(u_i)-1)=\int_\Omega \sum_{i=1}^N\ln(u_i)\p_t u_i.
\]
Then multiply Equation \eqref{eq:entropyform} by $\ln(u_i)$ and integrate by part in order to get:
\begin{equation*}
\int_\Omega\ln(u_i) \p_t u_i +\int_\Omega \left(\sum_{j=1}^N a_{i,j}u_i u_j\nabla \ln(u_i) \cdot \left(\nabla \ln(u_i)-\nabla \ln(u_j)\right)\right)=0.
\end{equation*}
Summing over $i\in\lent 1, N \rent$  yields the announced result thanks to the symmetry of $A$.
\end{proof}

The entropy - entropy dissipation relation~\eqref{eq:EDE} is key in the analysis of many 
cross-diffusion systems, as exposed in~\cite{Jungel15, Juengel16}. It will also 
play a central role in this paper. 
Assume that 
\be\label{eq:nondegenerate}
\min_{i\neq j} a_{i,j}>0 \ee
as it will be done in Section~\ref{sec:convergence}.
As a consequence of the inequality 
\begin{align*}
 \sum_{i=1}^N \int_\O |\nabla u_i|^2 \leq& \; 4   \sum_{i=1}^N \int_\O |\nabla \sqrt{u_i}|^2 \\
 \leq & \; 
 \frac{1}{\min_{i\neq j} a_{i,j}} \int_\Omega \sum_{1\leq i<j\leq N}a_{i,j}u_i u_j\left|\nabla \ln(u_i)-\nabla \ln(u_j)\right|^2,
\end{align*}
we deduce from~\eqref{eq:EDE} a $L^2(0,T;H^1(\O))^N$ estimate on $U$. This motivates the 
following notion of weak solution.

\begin{definition}\label{def:weaksol}
A \emph{weak solution $U$ to \eqref{eq:crossdiff} and~\eqref{eq:no-flux}} corresponding to the initial profile $U^0 \in L^\infty(\O;\Aa)$
 is a function of $L^\infty\left(Q_T; \Aa\right)\cap L^2\left([0,T];H^1(\O)\right)^N$  satisfying, $\forall i\in \lent 1, N\rent$, $\forall \varphi \in C^\infty_c([0,T)\times\overline{\Omega})$:
\begin{equation}\label{eq:wf}
{-}\iint_{Q_T} u_i \p_t \varphi \d x\d t - \int_\O u_i^0 \varphi(0,\cdot)\d x+ \iint_{Q_T}  \sum_{j=1}^N a_{i,j} \left(u_j \grad u_i-u_i \grad u_j\right) \grad\varphi=0.
\end{equation}
\end{definition}

The regularity requirement on a weak solution $U$ is natural in the setting where Assumption~\eqref{eq:nondegenerate} holds. 
In this case, the solution even enjoys a stronger regularity, as established in the recent contribution~\cite{berendsenStrongSolutionsWeakstrong2019}. In the case where~\eqref{eq:nondegenerate} 
is not fulfilled (but under a structural assumption on the matrix $A$), a more involved notion of weak solution 
has to be introduced, cf.~\cite{GJ18}.

There is an important property that relates the model~\eqref{eq:crossdiff} to classical Fickian diffusion. 
As a consequence of Lemma \ref{lem:admsum}, one can rewrite 
\begin{equation}\label{eq:changeofastar}
\dive\left(\sum_{j=1}^N \left(u_j\nabla u_i-u_i\nabla u_j\right)\right)=\lap u_i, \qquad i \in \lent 1, N \rent.
\end{equation}
As a consequence, if all the $a_{i,j}$ are equal to some $a \in \R$, then the system~\eqref{eq:crossdiff} reduces 
to $N$ uncoupled heat equations $\p_t u_i = a \Delta u_i$. 
Based on the identity~\eqref{eq:changeofastar}, we can rewrite the system~\eqref{eq:crossdiff} under the form
\begin{equation}\label{eq:crossdiffand heat}
\partial_t u_i -a^\star\lap u_i-\dive\left(\sum_{j=1}^N (a_{i,j}-a^\star)\left(u_j\nabla u_i-u_i\nabla u_j\right)\right)=0, \quad i \in \lent 1,N \rent,
\end{equation}
where $a^\star \in \R$ is arbitrary for the moment. The formulation~\eqref{eq:crossdiffand heat} is at the basis of our 
discretization.

\subsection{Objectives} \label{ssec:obj}
The goal of this paper is to build and analyze a numerical scheme preserving the properties discussed in the previous section, namely:
\begin{itemize}
\item the non-negativity of the concentrations;
\item the conservation of mass (Lemma~\ref{lem:masscons});
\item the preservation of the volume filling constraint (Lemma~\ref{lem:admsum});
\item the entropy-entropy dissipation relation (Proposition~\ref{prop:entropy}).
\end{itemize}

The construction of our scheme is the purpose of Section \ref{sec:FV}. In Section~\ref{sec:existence}, 
we will show the existence of solutions to this scheme and the preservation of discrete counterparts 
to the previously listed physical properties. Section~\ref{sec:convergence} is devoted to the convergence 
of the numerical scheme toward weak solutions provided Assumption~\eqref{eq:nondegenerate} is satisfied. 
Finally{,} in Section~\ref{sec:num}, we show the outcomes of some numerical experiments.

Before entering the core of the paper, let us mention that the development of 
numerical analysis for cross-diffusion systems is quite recent. To our 
knowledge, the first convergence study of a finite volume approximation for a non-degenerate 
cross-diffusion problem was carried out in~\cite{ABRB11}. This {contribution} is based on classical 
quadratic energy estimate{, similarly to what is proposed in~\cite{ZWYW18}}.
The implementation of the discrete entropy method~\cite{CH14_FVCA7} 
for cross-diffusion systems is more recent. Let us cite \cite{Ahmed_intrusion, Bourriquets} 
where upstream mobility finite volume and control volume finite {element} schemes for a 
multiphase extension of the porous medium equation are studied. 
Upwinding is also used in~\cite{Anita2019} to approximate the solution of a system which is very close to 
the problem~\eqref{eq:crossdiff} under study, or in~\cite{CFS_arXiv} for a problem in which nonlocal interactions 
are also considered. 
As a consequence of the upwind choice for the mobility, the schemes presented 
in~\cite{Ahmed_intrusion, Bourriquets, Anita2019} and \cite{CFS_arXiv} are first{-}order accurate in space.
{A} natural solution to pass to order two is to rather consider mobilities given by 
arithmetic means~\cite{DJZ_arXiv}. The motivation of the finite element scheme 
proposed in~\cite{JL19} is also the same. However, the scheme proposed in~\cite{JL19} 
is expressed in entropy (or dual) variables (in our context $\log(u_i)$) leading to computational difficulties 
when the concentrations are close to $0$. 
{Other entropy stable numerical schemes have been proposed for cross-diffusion systems, as for instance
discontinuous Galerkin schemes in~\cite{SCS19}, or finite volumes on staggered cartesian grids for 
Maxwell-Stefan cross-diffusion in~\cite{HLTW_arXiv}. 
Finally, let us point out that the design of entropy (or energy) stable numerical schemes for dissipative systems 
with formal gradient flow structure in a Riemannian geometry have been the purpose of intense research 
in the recent years, as shows the extensive (but not exhaustive) 
recent literature \cite{CG_VAGNL, KSW18, CNV_HAL, ABPP19, MW19, CCFG_HAL, CN_HAL, GS20, BCMS_arXiv, CMW_arXiv} on this topic. 
Let us also refer to~\cite{CWWW19, DWZZ19, SXY19,  BKNR_arXiv, SX_preprint} for the simpler situation of gradient flows 
in Hilbert spaces.}

\begin{remark}\label{rmk:reac}
{Our study can be extended to the case where reaction terms are incorporated in the system. 
More precisely, one can consider a system of the form 
\be\label{eq:sys.reac}
\p_t U - \div(D(U) \grad U) = R(U), 
\ee
where $D(U)$ is as in~\eqref{eq:crossdiff2}, and where the reaction term function 
\[
R: \begin{cases}
\R^N \to \R^N\\
U \mapsto R(U) =  \left(r_i(U)\right)_{1\leq i \leq N}
\end{cases}
\]
is continuous and satisfies the following structural properties which are classically satisfied for 
reactive systems:
\begin{enumerate}[(i)]
\item\label{it:reac.iso} {\em Isochore process:} $\sum_{i=1}^N {r}_i(U) = 0$ for all $U \in \R^N$; 
\item\label{it:reac.pos} {\em Positivity preservation:} $r_i(U) \geq 0$ for $U \in \R^N$ with  $u_i \leq 0$; 
\item\label{it:reac.ent} {\em Entropy dissipation:} there exists $\ov U = \left(\ov u_i \right)_{1\leq i \leq N}$ in $\Aa$ with $\ov u_i >0$ 
for all $i\in\{1,\dot, N\}$ such that 
\be\label{eq:reac-dissip}
R(U) \cdot \ln(U/\ov U) = \sum_{i=1}^N r_i(U) \left(\ln(u_i) - \ln(\ov u_i)\right) \leq 0, \qquad \forall U \in \Aa.
\ee
\end{enumerate}
Because of reaction terms, the volume of each specie is no longer conserved, so that Lemma~\ref{lem:masscons} 
does no longer hold true. However, because of Assumption \eqref{it:reac.iso} above, the total volume is conserved, 
hence the condition $\sum_{i=1}^N u_i(t,x) =  1$ remains true for all time. Since Assumption~\eqref{it:reac.pos} guarantees 
the positivity of the solution, one gets that $U(t,x)$ belongs to $L^\infty(Q_T, \Aa)$ with the reaction term as well. 
Finally, Assumption~\eqref{it:reac.ent} on the reaction terms ensures that the relative entropy 
\be\label{eq:ent.rel}
E(U|\ov U)=\sum_{i=1}^N\int_\O u_i\log\left(\frac{u_i}{\ov u_i}\right) \geq0
\ee
is a Lyapunov functional for the system. This stability property allows to extend our purpose in presence 
of reaction.
Note that in absence of reaction $R \equiv 0$, 
this relative entropy (one can for instance set $\ov u_i = \oint_\O u_i^0$) coincides with the mixing 
entropy~\eqref{eq:ent.mix} up to an additive constant 
thanks to the conservation of the volume of each specie, cf. Lemma~\ref{lem:masscons}.}

{Finally, let us note that if $\ov U \in \Aa$ is such that $\ov u_i = 0$ for some $i \in \{1,\dots, N\}$, the relative entropy
$E(U|\ov U)$ is no longer well-defined. Our analysis can still be extended by showing that the mixing entropy $E(U)$ grows 
at most linearly with time, which is sufficient for establishing the convergence of the straightforward extension to the 
case $R\neq 0$ of the finite volume scheme to be presented in the next section.}
\end{remark}

\section{Finite Volume approximation}\label{sec:FV}
	This section is organized as follows. First, in Section~\ref{ssec:mesh}, 
we state the requirements on the mesh and fix some notations. Then in Section~\ref{ssec:scheme}, 
we describe the numerical scheme to be studied in this paper. It is based on Formulation \eqref{eq:crossdiffand heat} 
of the problem.
Then in Section~\ref{ssec:mainthms}, 
we state our two main results. The first one, namely Theorem~\ref{thm:main1}, focuses on 
the case of a fixed mesh. We are interested 
in the existence of a solution to the nonlinear system corresponding to the scheme, and 
the dissipation of the entropy at the discrete level. More precisely, one establishes 
that the studied scheme satisfies a discrete entropy - entropy dissipation inequality 
that {can} be thought of as a counterpart to Proposition~\ref{prop:entropy}.
Our second main result, namely Theorem~\ref{thm:main2}, is devoted to the convergence of the 
scheme towards a weak solution as the time step and the mesh size tend to $0$. 

% ============================================================
\subsection{Discretization of $(0,T) \times \O$}\label{ssec:mesh}
The scheme we propose relies on two-point flux approximation (TPFA) finite volumes.
As explained 
in~\cite{Droniou-review,Tipi,gartnerWhyWeNeed2019},
 this approach appears to be very efficient as soon 
as the continuous problem to be solved numerically is isotropic and one has the 
freedom to choose a suitable mesh fulfilling the so-called orthogonality condition~\cite{Herbin95, EGH00}.
We recall here the definition of such a mesh, which is illustrated in Figure~\ref{fig:mesh}.

\begin{definition}
\label{def:mesh}
An \emph{admissible mesh of $\O$} is a triplet $\left(\Tt, \Ee, {(x_K)}_{K\in\Tt}\right)$ such that the following conditions are fulfilled. 
\begin{enumerate}[(i)]
\item Each control volume (or cell) $K\in\Tt$ is non-empty, open, polyhedral and convex. We assume that 
\[
K \cap L = \emptyset \quad \text{if}\; K, L \in \Tt \; \text{with}\; K \neq L, 
\qquad \text{while}\quad \bigcup_{K\in\Tt}\ov K = \ov \O. 
\]
\item Each face $\sig \in \Ee$ is closed and is contained in a hyperplane of $\R^d$, with positive 
$(d-1)$-dimensional Hausdorff (or Lebesgue) measure denoted by $m_\sig = \cH^{d-1}(\sig) >0$.
We assume that $\cH^{d-1}(\sig \cap \sig') = 0$ for $\sig, \sig' \in \Ee$ unless $\sig' = \sig$.
For all $K \in \Tt$, we assume that 
there exists a subset $\Ee_K$ of $\Ee$ such that $\p K =  \bigcup_{\sig \in \Ee_K} \sig$. 
Moreover, we suppose that $\bigcup_{K\in\Tt} \Ee_K = \Ee$.
Given two distinct control volumes $K,L\in\Tt$, the intersection $\ov K \cap \ov L$ either reduces to a single face
$\sig  \in \Ee$ denoted by $K|L$, or its $(d-1)$-dimensional Hausdorff measure is $0$. 
\item The cell-centers $(x_K)_{K\in\Tt}$ satisfy $x_K \in K$, and are such that, if $K, L \in \Tt$ 
share a face $K|L$, then the vector $x_L-x_K$ is orthogonal to $K|L$.
\item For the boundary faces $\sig \subset \p\O$, we assume that either $\sig \subset \Gamma_D$ or $\sig \subset \ov \Gamma_N$.
For $\sig \subset \p\O$ with $\sig \in \Ee_K$ for some $K\in \Tt$, we assume additionally that 
there exists $x_\sig \in \sig$ such that $x_\sig - x_K$ is orthogonal to $\sig$.

\end{enumerate} 
\end{definition}

\begin{figure}[htb]
\centering
\resizebox{7cm}{!}{\input{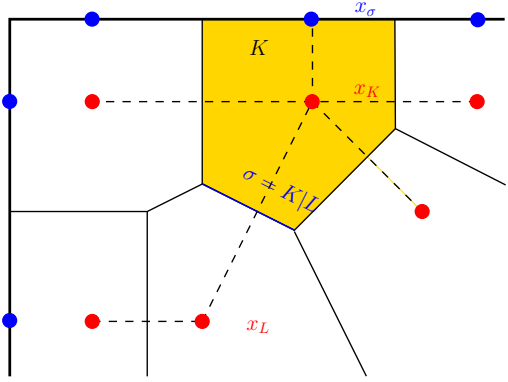}}
\caption{Illustration of an admissible mesh as in Definition~\ref{def:mesh}.}
\label{fig:mesh}
\end{figure}
We denote by $m_K$ the $d$-dimensional Lebesgue measure of the control volume $K$.
The set of the faces is partitioned into two subsets: the set $\Ee_{\rm int}$ of the interior faces defined by 
$
\Ee_{\rm int}= \left\{ \sig \in \Ee\; \middle| \; \sig = K|L\; \text{for some}\; K,L \in \Tt\right\}, 
$
and the set $\Ee_{\rm ext}$ of the exterior faces defined by 
$
\Ee_{\rm ext}= \left\{ \sig \in \Ee\; \middle| \; \sig \subset \p\O\right\}
$. For a given control volume $K\in\Tt$, we also define $\Ee_{K,{\rm int}} = \Ee_K \cap \Ee_{\rm int}$ the set of its faces that belong to $\Ee_{\rm int}$. For such a face $\sigma\in \Ee_{K,{\rm int}}$, we may write $\sigma=K|L$, meaning that $\sigma=\ov K \cap \ov L$, where $L\in\Tt$ .

Given $\sig \in \Ee$, we let
\[
d_\sig = \begin{cases}
|x_K - x_L| & \text{if}\; \sig = K|L \in \Ee_{\rm int}, \\
|x_K - x_\sig| & \text{if}\; \sig \in \Ee_{\rm ext}, 
\end{cases}
\qquad 
\text{and}\quad \tau_\sig = 
\frac{m_\sig}{d_\sig}.
\]
We finally introduce the size $h_\Tt$ and the regularity $\zeta_\Tt$ (which is assumed to be positive) of a discretization $(\Tt, \Ee, (x_K)_{K\in\Tt})$ of $\O$ by setting 
\[
h_\Tt = \max_{K \in \Tt}\;{\rm diam}(K), \qquad \zeta_\Tt = \min_{K\in\Tt}\min_{\sig\in\Ee_K} \frac{d(x_K,\sigma)}{d_\sig}.
\]

Concerning the time discretization of $(0,T)$, we consider an increasing finite family of times $0 = t_0 < t_1 < \dots, < t_{N_T} = T$. 
We denote by $\dt_n = t_{n}-t_{n-1}$ for $n\in\{1,\dots, N_T\}$, by $\boldsymbol{\Delta t} = \left(\dt_n\right)_{1\leq n \leq N_T}$, 
and by $h_T = \max_{1 \leq n \leq N_T} \dt_n$. 
In what follows, we will use boldface notations for mesh-indexed families, typically for elements of $\R^\Tt, {(\R^{\Tt})}^N$, 
${(\R^{\Tt})}^{N_T}$, or even ${(\R^\Tt)}^{N\times N_T}$.

\subsection{Numerical scheme}
\label{ssec:scheme}

The initial data $U^0\in L^\infty(\O;\Aa)$ is discretized into 
\[
\bU^0 = \left(\bu_i^0\right)_{i\in\lent 1, N \rent} \in {(\RR^\Tt)}^N = \left(u^0_{i,K}\right)_{K\in\Tt, i \in \lent 1, N \rent}
\] 
by setting
\begin{equation}\label{eq:cK0}
u_{i,K}^0 = \frac1{m_K} \int_K u_i^0(x) \d x, \qquad \forall K \in \Tt, i\in\lent1, N\rent. 
\end{equation}
Assume that $\bU^{n-1}=\left(u_{i,K}^{n-1}\right)_{K\in\Tt,i\in\lent1,N\rent}$ is given for some $n \geq 1$, 
then we have to define how to compute $\bU^n=\left(u_{i,K}^{n}\right)_{K\in\Tt,i\in\lent1,N\rent}$. 

First, we introduce some notations. Given any discrete scalar field $\bc=\left(c_{K}\right)_{K\in\Tt}\in \R^\Tt$, we define for all cell $K\in \Tt$ and interface $\sig \in \Ee_K$ the mirror value $c_{K\sig}$ of $c_{K}$ across $\sig$ by setting:
\begin{equation}\label{eq:mirror}
c_{K\sig} = \begin{cases}
c_{L}  &\text{if}\; \sig = K|L \in \Ee_{\rm int},\\
c_{K} & \text{if}\; \sig \in \Ee_{\rm ext}.
\end{cases}
\end{equation}
We also define the oriented and absolute jumps of $\bc$ across any edge by 
\[
D_{K\sig}\bc = c_{K\sig} - c_K, \quad D_\sig\bc = |D_{K\sig} \bc|, \qquad \forall K \in \Tt, \; \forall \sig \in \Ee_K.
\] 

The scheme is based on the formulation \eqref{eq:crossdiffand heat}. 
It requires the introduction of a parameter $a^\star$ on which we only have the following requirements:
\begin{equation}\label{eq:choiceofa}
a^\star>0 \qquad\text{and}\qquad a^\star\geq \min_{(i,j)}a_{i,j}.
\end{equation}
\begin{subequations}\label{scheme}
The conservation laws are discretized in a conservative way with a time discretization relying 
on the backward Euler scheme:
\begin{equation}\label{eq:scheme_c}
m_K\frac{u_{i,K}^n - u_{i,K}^{n-1}}{\dt_n}  + \sum_{\sig\in\Ee_K}F_{i,K\sig}^n = 0, \quad \forall K \in \Tt, \; \forall i \in \lent 1, N \rent.
\end{equation}
The discrete fluxes are computed thanks to a formula based on~\eqref{eq:crossdiffand heat} 
and on TPFA finite volumes:
\begin{equation}\label{eq:flux}
F_{i,K\sig}^n=-a^\star\tau_\sig D_{K\sig}\bu_i^n -\tau_\sigma\left(\sum_{j=1}^N (a_{i,j}-a^\star)\left(u_{j,\sigma}^nD_{K\sigma} u_i^n-u_{i,\sigma}^nD_{K\sigma} u_j^n\right)\right),
\end{equation}
for all $K\in\Tt$, $\sig\in\Ee_K$ and $i\in\lent1,N\rent$. 
Edge values $\left(u_{j,\sig}^n\right)_j $ of the concentrations $u_j$ appears in Formula~\eqref{eq:flux}. 
It is deduced from $u_{j,K}^n$ and $u_{j,K\sig}^n$ thanks to a logarithmic mean, i.e., 
\begin{equation}\label{eq:avg}
u_{j,\sigma}^n=\begin{cases}
0&\text{if }\min(u_{j,K}^n,u_{j,K\sigma}^n)\leq0,\\
u_{j,K}^n&\text{if } 0\leq u_{j,K}^n=u_{j,K\sigma}^n,\\
\dfrac{u_{j,K}^n-u_{j,K\sigma}^n}{\ln(u_{j,K}^n)-\ln(u_{j,K\sigma}^n)}&\text{otherwise}.
\end{cases}
\end{equation}
\end{subequations}
This choice for the edge concentration is crucial for the preservation at the discrete level of a discrete entropy - entropy dissipation inequality similar to the one highlighted in Proposition \ref{prop:entropy}. Equations \eqref{eq:flux} and \eqref{eq:mirror} implies that for all $\sig\in\Ee_{\rm ext}$: $F_{i,K\sig}^n=0$, so that the no-flux 
boundary condition~\eqref{eq:no-flux} is taken into account. 

\begin{remark}\label{rmk:astar}
Let us highlight why the choice of a strictly positive $a^\star$ is important. 
Consider a mesh with two cells $K,L$, and one edge. We consider two species 
and let $u^0_K=(0,1)$ and $u^0_L=(1,0)$. We have: $u_{1,K|L}^0=0$ and 
$u_{2,K|L}^0=0$, hence, if $a^\star=0$, the initial condition is a stationary solution 
even though this is not expected for a discretization of the heat equation. 
Setting $a^\star>0$ eliminates these spurious solutions.
The choice of $a^\star$ has a strong influence on the numerical outcomes, 
as it will be shown in Section~\ref{sec:num}, but we don't have a clear understanding yet 
on the methodology to choose an optimal $a^\star$. What {seems} clear is that $a^\star$ has to be chosen 
in the interval $[\min_{i\neq j} a_{i,j}, \max_{i\neq j} a_{i,j}]$. A tentative non-optimal formula is 
proposed in Section~\ref{sec:num}.
\end{remark}

{
\begin{remark}\label{rmk:higher-order}
The time discretization in scheme~\eqref{scheme} is only first{-}order accurate since it relies on 
the backward Euler approximation. Going to second{-}order time dis{c}retizations is tempting, 
but no theoretical guarantees concerning the entropy stability of the scheme can be granted then. 
This is due to the fact that the entropy~\eqref{eq:ent.mix} is not quadratic, hence neither 
the Crank-Nicolson scheme nor the BDF2 scheme can be shown to be unconditionally stable here. 
This lack of theoretical foundation for the entropy stability has for instance also been reported in
\cite{GS20}.
\end{remark}
}

\subsection{Main results and organization}\label{ssec:mainthms}

The first theorem proven is this paper concerns the existence of discrete solutions for a given mesh, 
and the preservation of the structural properties listed in Section \ref{ssec:obj}:
\begin{itemize}
\item the mass of each specie is conserved along the time steps;
\item the concentrations are (strictly) positive and sum to $1$ in all the cells, i.e., $U_K^n\in\Aa$ for all $K\in\Tt$ and $n\geq 1$; 
\item the discrete counterpart of the entropy decays along time.
\end{itemize}
For this last property, we need to introduce the discrete entropy functional $E_\Tt$, which is defined by:
\begin{equation}\label{eq:E_T}
E_\Tt(\bU) = \sum_{K\in\Tt}\sum_{i=1}^N m_K u_{i,K} \ln u_{i,K}, 
\qquad \forall \bU = \left(u_{i,K}\right)_{K\in\Tt, i\in\lent 1,N\rent} \in \Aa^\Tt. 
\end{equation}
As stated in Theorem~\ref{thm:main1} below, the nonlinear system corresponding to our scheme~\eqref{scheme}
admits solutions that preserve the physical bounds on the concentrations and the decay of the entropy. 

\begin{theorem}\label{thm:main1}
Let $(\Tt,\Ee,\left(x_K\right)_{K\in\Tt})$ be an admissible mesh and let $\bU^0$ be defined by~\eqref{eq:cK0}. Then, for all $1\leq n\leq N_T$, the nonlinear system of equations \eqref{eq:mirror} -- \eqref{scheme},  has a positive solution $\bU^{n}\in \Aa^{\Tt}$.
Moreover, such a solution satisfies $\ds E_\Tt(\bU^{n}) \leq E_\Tt(\bU^{n-1})$ for all $n\in \lent 1,N_T \rent$,
 $\sum_{K\in\Tt}m_K u_{i,K}^n=\int_\Omega u_i^0$ for all $i\in \lent 1,N \rent$ and $n\in \lent 0,N_T \rent$.
\end{theorem}
The proof of Theorem~\ref{thm:main1} will be the purpose of Section~\ref{sec:existence}.
With a discrete solution  $(\bU^{n})_{1\leq n\leq N_T}$ to the scheme~\eqref{scheme} at hand,
we can define the piecewise constant approximate solution $U_{\Tt,\bdt} = \left(u_{i,\Tt,\bdt}\right)_{i\in\lent 1,N\rent}: Q_T \to \Aa$
defined almost everywhere by 
\[
U_{\Tt, \boldsymbol{\Delta t}}(t,x) = U_K^{n} \quad \text{if}\;(t,x) \in (t_{n-1},t_n]\times K.
\]
This definition will be developed in Section \ref{sec:convergence} and supplemented by other reconstruction operators.
Let $\left(\Tt_m, \Ee_m, \left(x_K\right)_{K\in\Tt_m}\right)_{m\geq 1}$ be a sequence
of admissible discretizations with $h_{\Tt_m}, h_{T,m}$ tending to $0$ as $m$ tends to $+\infty$, 
while the regularity $\zeta_{\Tt_m}$ remains uniformly bounded from below by a positive constant $\zeta^\star$. Thanks to Theorem~\ref{thm:main1}, we dispose of a family $\bU_m$ of solutions to our scheme. The convergence of $\bU_m$ is the purpose of Theorem~\ref{thm:main2} whose proof is detailed in Section~\ref{sec:convergence}.
\begin{theorem}\label{thm:main2}
Assume that the nondegeneracy assumption~\eqref{eq:nondegenerate} holds.
Given any sequence of solutions $\bU_m= \left(u_{i,K}^n\right)_{i\in \lent 1, N\rent, K\in\Tt_m, 1 \leq n \leq N_{T,m}}$, there exists at least one $U \in L^\infty\left(Q_T; \Aa\right) \cap L^2\left((0,T);H^1(\O)\right)$ such that, up to a subsequence,
\begin{equation}\label{eq:conv_Lp}
\bU_{\Tt_m, \boldsymbol{\Delta t}_m}  \underset{m\to\infty} \longrightarrow U \quad \text{strongly in $L^p(Q_T)$, for any $1\leq p <\infty$},\end{equation}
Moreover, $U$ is a weak solution in the sense of Definition~\ref{def:weaksol}.
\end{theorem}

{
\begin{remark}\label{rmk:error}
Theorem~\ref{thm:main2} above establishes the convergence of the scheme, but no error estimate can be deduced 
from its proof. Indeed, its proof can be thought of as an adaptation to the discrete setting of an existence proof based 
on compactness arguments, as the for instance proposed in~\cite{Jungel15}. The derivation of error estimates is 
different since it relies on the perturbation of uniqueness proofs. As far as we know, uniqueness for the problem 
under consideration is an open question excepted in the one-dimensional setting~\cite{berendsenStrongSolutionsWeakstrong2019}, 
but a natural approach would be to derive error estimates based on the use of the relative entropy, as done for instance 
in~\cite{CMS16} for hyperbolic systems or in~\cite{GHMN16, GMN19} for the compressible Navier-Stokes problem.
However, the aforementioned strategy generally leads to under-optimal error estimates. 
The recovery of optimal error estimates for finite volume approximation of diffusion equations on unstructured grids 
has only been achieved recently \cite{DN18}. The extension of this optimal result to our much more complex cross-diffusion 
system appears to be an interesting and challenging issue. 
\end{remark}
}

\section{Numerical analysis on a fixed mesh}\label{sec:existence}

This section is devoted to the proof of Theorem~\ref{thm:main1}. In Section~\ref{ssec:apriori}, 
we establish a priori estimates on a slightly modified scheme that will be shown to 
reduce to the original scheme~\eqref{scheme}. 
Then in Section~\ref{ssec:existence}, we apply a topological degree argument 
to prove the existence of solutions to our scheme. 
Section~\ref{ssec:entropy} is devoted to the proof of the entropy dissipation property.

To prove the existence of solutions to the system of equations \eqref{scheme}, we need the inequality $\sum_i u_{i,\sigma}\leq 1$. We then slightly modify \eqref{scheme} by adding the following equation:
\[
\wt{u}_{i,\sigma}^n=\dfrac{u_{i,\sigma}^n}{\max(1,\sum_{j=1}^Nu_{j,\sigma}^n)},
\]
and replacing $u_{i,\sigma}^n$ by $\wt{u}_{i,\sigma}^n$ in \eqref{eq:flux}. We will denote this new system (S) and see in Proposition~\ref{prop:systemeq} that its solutions satisfy $\sum_i u_{i,\sigma}\leq 1$, so that $\wt{u}_{i,\sigma}^n=u_{i,\sigma}^n$. Whence they also satisfy the original system of equations.

\subsection{A priori estimates}\label{ssec:apriori}

The first lemma shows the nonnegativity of the solutions to (S). 
\begin{lemma}\label{lem:nonnegativity}
Given a nonnegative $\bU^{n-1}$, any solution $\bU^n$ to (S) is also nonnegative.
\end{lemma}
\begin{proof}
 Let $\bU^n$ be a solution of (S) and let $i\in\lent 1, N\rent$. We consider a cell $K\in\Tt$ where $\bu_i^n$ reaches {its} minimum, 
 i.e., $u_{i,K}^n \leq u_{i,L}^n$ for all $L \in \Tt$, and assume for contradiction that $u_{i,K}^n$ is (strictly) negative. 
 Equation \eqref{eq:flux} then gives:
\[
 m_K\frac{u_{i,K}^n - u_{i,K}^{n-1}}{\dt_n}  =- \sum_{\sig\in\Ee_K}F_{K\sig}^n.
\]
The term on the left hand side is negative since $u_{i,K}^{n-1} \geq 0 > u_{i,K}^n$, 
whereas the right-hand side may be simplified noticing that $\wt{u}_{i,\sigma}^n=0$:
\[
\sum_{\sig\in\Ee_K}a^\star\tau_\sig D_{K\sig}\bu_i^n +\sum_{\sig\in\Ee_K} \tau_\sigma\sum_{j=1}^N (a_{i,j}-a^\star)\wt{u}_{j,\sigma}^nD_{K\sigma} \bu_i^n=- \sum_{\sig\in\Ee_K}F_{K\sig}^n<0.
\]
Noticing that $D_{K\sig}\bu_i^n\geq 0$, $\wt{u}_{j,\sigma}^n\geq 0$, 
and $\sum_{j=1}^N\wt{u}_{j,\sigma}^n\leq 1$ we obtain that 
\[
0\leq \sum_{\sig\in\Ee_K}a^\star(1-\sum_{j=1}^N\wt{u}_{j,\sigma}^n)\tau_\sig D_{K\sig}\bu_i^n< 0,
\]
which is absurd, hence the desired result.
\end{proof}

Let us now show that the concentrations sum to $1$ in all the cells. 
\begin{lemma}\label{lem:admcond}
Given $\bU^{n-1}$ in $\Aa^\Tt$ , any solution $\bU^n$ to (S) is also in $\Aa^\Tt$.
\end{lemma}
\begin{proof}
Thanks to Lemma~\ref{lem:nonnegativity}, it suffices to show that
$\sum_{i=1}^N u_{i,K}^n=1$ for all $K \in \Tt.$
Let $\bU^n$ be a solution to (S). Using \eqref{eq:flux} in \eqref{eq:scheme_c} and summing over the species 
leads to:
\begin{multline*}
\dfrac{\sum_{i=1}^N u^n_{i,K}-\sum_{i=1}^N u^{n-1}_{i,K}}{\Delta t_n} m_K
-a^\star\sum_{\sigma\in \Ee_K} \tau_\sigma D_{K\sigma}\sum_i \bu_i\\
-\sum_{\sigma\in \Ee_K}\tau_\sigma\sum_i\left(\sum_{j=1}^N (a_{i,j}-a^\star)\left(\wt{u}_{j,\sigma}^nD_{K\sigma} \bu_i-\wt{u}_{i,\sigma}^nD_{K\sigma} \bu_j\right)\right)=0, 
\quad \forall K \in \Tt.
\end{multline*}
The third term of the left-hand side vanishes thanks to the symmetry of $A$, so that 
\begin{equation*}
\dfrac{\sum_{i=1}^N u^n_{i,K}-\sum_{i=1}^N u^{n-1}_{i,K}}{\Delta t_n} m_K
-a^\star\sum_{\sigma\in \Ee_K} \tau_\sigma D_{K\sigma}\sum_i \bu_i
=0, \quad \forall K \in \Tt.
\end{equation*}
The discrete quantity $\sum_i \bu_i$ is solution to the classical backward Euler TPFA scheme for the heat equation,
which is well posed. So $\sum_i\bu_i^n=\sum_i\bu^{n-1}_i=\boldsymbol{1}$ is its unique solution, hence the desired result.
\end{proof}

\subsection{Existence of solutions}\label{ssec:existence}
Using the tools exposed in the previous subsection, we may derive the existence of a solution to (S):
\begin{prop}
 Given $\bU^{n-1}$ in $\Aa^\Tt$, there exists at least one solution to (S) in $\Aa^\Tt$.
\end{prop}
\begin{proof}
The proof relies on a topological degree argument~\cite{LS34,Dei85}. 
The idea is to transform continuously our complex nonlinear system into a linear system 
while guaranteeing that the a priori estimates controlling the 
solution remain valid all along the homotopy. We sketch the main ideas of the proof, 
making the homotopy explicit. We are interested in the existence of zeros for a functional 
\[
\bHh: \begin{cases}
[0,1] \times (\R^N)^\Tt  \to (\R^N)^\Tt\\
(\lambda, \bU) \mapsto \bHh(\lambda, \bU)
\end{cases}
\]
that boils down to the scheme (S) when $\lambda = 1$. In our case, we set:
\begin{multline}
\bHh(\lambda, \bU)_{i,K}= \dfrac{ u_{i,K}- u^{n-1}_{i,K}}{\Delta t_n} m_K
-a^\star\sum_{\sigma\in \Ee_K} \tau_\sigma D_{K\sigma} \bu_i\\
-\lambda\sum_{\sigma\in \Ee_K}\tau_\sigma\left(\sum_{j=1}^N (a_{i,j}-a^\star)\left(\wt{u}_{j,\sigma}D_{K\sigma} \bu_i-\wt{u}_{i,\sigma}D_{K\sigma} \bu_j\right)\right), \quad \forall K\in\Tt, \forall i\in\lent 1, N\rent.
\end{multline}
One notices that $\bHh(0, \bU)={\bf 0}$ is the classical heat equation, the solution of which belongs to $\Aa^\Tt$.  
Therefore, fixing $\eta>0$, the relatively compact open set 
\[
\Aa^\Tt_\eta = \left\{\; \bU \in \R^{\Tt} \;\middle| \; \inf_{\boldsymbol{V} \in \Aa^\Tt} \| \bU - \boldsymbol V\| < \eta\;\right\}
\] 
has a topological degree equal to $1$. 
Note that the choice of the norm in the definition of $\Aa^\Tt_\eta$ is not important since the dimension is finite. 
Moreover, thanks to Lemma \ref{lem:admcond}, the solutions $\bu^{(\lambda)}$ of $\bHh(\lambda,\bU)={\bf 0}$ remains in $\Aa^\Tt$, 
thus in the interior of $\Aa^\Tt_\eta$. Thus the topological degree of $\Aa^\Tt_\eta$ for $\lambda=1$ 
is still equal to $1$, hence the existence of (at least) one solutions to (S). Since $\eta>0$ is arbitrary, then there is a solution 
in $\Aa^\Tt = \bigcap_{\eta>0} \Aa^\Tt_\eta$.
\end{proof}

To prove the Theorem~\ref{thm:main1}, we need to transfer this {existence} result on the original system.
\begin{prop}\label{prop:systemeq}
	A solution $\bU^n$ of (S) is a solution of \eqref{scheme}.
	Reciprocally, a solution of \eqref{scheme} in $\Aa^\Tt$ is a solution of (S).
\end{prop}
\begin{proof}
Let $\bU^n$ be a solution of (S). A simple convexity argument shows that the logarithmic mean of two nonnegative number 
is smaller than the arithmetic mean, so that $u_{i,\sigma}^n\leq \frac{u_{i,K}^n+u_{i,K\sigma}^n}2$. 
Summing w.r.t. $i\in\lent 1, N \rent$ and using that the solution $\bU$ of (S) belongs to $\Aa^\Tt$, one gets that 
\(
 \sum_i u_{i,\sigma}^n\leq 1
 \)
 for all $\sig \in \Ee$. Therefore $\wt{u}_{i,\sigma}^n=u_{i,\sigma}^n$ and $\bU^n$ is a also solution to \eqref{scheme}. 
 The proof of the reverse implication follows the same lines.
\end{proof}

\subsection{Entropy dissipation}\label{ssec:entropy}
We intend here to prove a discrete counterpart to Proposition~\ref{prop:entropy}. The proof will be very similar and requires
a discrete counterpart of the conservation of mass (Lemma~\ref{lem:masscons}).
\begin{lemma}\label{lem:conservation}
Given any $\bU^{n-1}\in\Aa^\Tt$, any solution $\bU^n$ to \eqref{scheme} satisfies:
\[
\sum_{K\in\Tt}m_K u_{i,K}^n=\sum_{K\in\Tt}m_K u_{i,K}^{n-1} = \int_\O u_i^0 \d x, \quad \forall i \in \lent 1,N\rent.
\]
\end{lemma}
The proof of this lemma is a straightforward calculation based on equation \eqref{eq:scheme_c}, the conservativity of the fluxes, 
and the definition~\eqref{eq:cK0} of the discrete initial condition.
With this lemma and Proposition~\ref{prop:systemeq}, we can refine the result Lemma~\ref{lem:nonnegativity} 
to get the strict positivity of any solution to \eqref{scheme} belonging to $\Aa^\Tt$.
\begin{lemma}\label{lem:positivity}
Let $\bU^{n-1}\in\Aa^\Tt$ be such that $\sum_K m_K u_{i,K}^{n-1}>0$ for all $i\in\lent 1, N\rent$, 
then any solution to \eqref{scheme} in $\Aa^\Tt$ is positive: $u^n_{i,K}>0$ for all $i\in\lent 1, N\rent$ and all $K\in\Tt$.
\end{lemma}
\begin{proof}
Let $\bU^n\in \Aa^\Tt$ be a solution to the scheme \eqref{scheme}, and let $i\in \lent1,N\rent$. 
We know from Lemma~\ref{lem:nonnegativity} that $\bu_i^n\geq {\bf 0}$. 
Assume for contradiction that there exists  one cell $K$ such that $u_{i,K}^0$ vanishes.
Using Lemma~\ref{lem:conservation} and the connectivity of $\Omega$, there exists $\sigma=K|L\in\Ee^\text{int}$ such that $u_{i,K}^n=0$ and $u_{i,L}^n>0$. Then $u_{i,\sigma}^n=0$ and as in the proof of Lemma~\ref{lem:nonnegativity}:
\[
a^\star(1-\sum_{j=1}^N {u_{j,\sigma}^n})\tau_\sig D_{K\sig}\bu_i^n\leq 0.
\]
Using $u_{j,\sig}^n \leq \frac{u_{j,K}^n + u_{j,L}^n}2$ and $u_{i,\sig}=0$ we deduce that 
\[
\sum_{j=1}^N {u_{j,\sigma}^n}\leq \sum_{j\neq i}^N \frac{u_{j,K}^n+u_{j,L}^n}{2}\leq 1-\frac{u_{i,L}^n}2<1. 
\]
Therefore 
\(
a^\star(1-\sum_{j=1}^N\wt{u}_{j,\sigma}^n)\tau_\sig>0, 
\)
and since $D_{K\sig}\bu_i^n>0$, we deduce that:
\[
0<a^\star(1-\sum_{j=1}^N\wt{u}_{j,\sigma}^n)\tau_\sig D_{K\sig}\bu_i^n\leq 0.
\]
As this statement is absurd, our assumption was false, hence the desired result.
\end{proof}
As in the continuous case, we will use the conservation of mass (Lemma~\ref{lem:conservation}) 
and a discrete equivalent of the chain rule $\nabla c=c\nabla\ln c$. This {equivalently} writes 
\be\label{eq:chainrule}
D_{K\sigma}\bu_i^n=u_{i,\sigma}^n D_{K\sigma}\ln(\bu_i^n), \quad \forall i \in \lent 1,N\rent, \; \forall K \in \Tt.
\ee
The above discrete chain rule follows from 
the definition~\eqref{eq:avg} of $u_{i,\sigma}^n$ and the positivity of solutions to \eqref{scheme}
which gives a sense to $\ln(\bu_i^n)$. 

Using~\eqref{eq:chainrule} in~\eqref{eq:flux}, $\bU^n$ satisfies
\begin{multline}\label{eq:discreteCDH-lnform}
   \frac{u^n_{i,K}-u^{n-1}_{i,K}}{\Delta t_n}m_K
-\sum_{\sigma\in \Ee_K}\tau_\sigma\left(\sum_{j=1}^N (a_{i,j}-a^\star)u_{i,\sigma}^nu_{j,\sigma}^n
\left(D_{K\sigma} \ln(\bu_i)-D_{K\sigma} \ln(\bu_j)\right)\right)\\-a^\star\sum_{\sigma\in \Ee_K} 
\tau_\sigma D_{K\sigma}\bu_i^n=0, \quad \forall K\in\Tt,\forall i\in\lent1,N\rent.
\end{multline}
This reformulation is suitable for proving a discrete entropy - entropy dissipation inequality, 
which should be seen as a discrete counterpart of Proposition~\ref{prop:entropy}.
\begin{prop}\label{prop:EEDdiscrete}
	Given $\bU^{n-1}$ in $\Aa^\Tt$, any solution $\bU^n\in \Aa^\Tt$ to~\eqref{scheme} satisfies 
\begin{equation}\label{eq:EEDdiscrete}
	E_\Tt(\bU^n)- E_\Tt(\bU^{n-1}) 
	+\Delta t_n \min_{1\leq i,j\leq N} a_{i,j}\sum_{\sigma\in \Ee}\sum_{i=1}^N \tau_\sigma u_{i,\sigma}^n(D_{K\sigma}\ln(\bu_i^n))^2\leq 0.
\end{equation}
In particular, $E_\Tt(\bU^n) \leq E_\Tt(\bU^{n-1})$.
\end{prop}
\begin{proof}
Multiplying equation \eqref{eq:discreteCDH-lnform} by $\Delta t_n\ln(u_{i,K}^n)$ and summing over the cells and species leads to:
\begin{multline}\label{eq:primalEED}
\sum_{K\in \Tt} \sum_{i=1}^N (u_{i,K}^n\ln(u_{i,K}^n)-u_{i,K}^{n-1}\ln(u_{i,K}^n))m_K
+\Delta t_n a^\star\sum_{\sigma\in \Ee}\sum_{i=1}^N \tau_\sigma u_{i,\sigma}^n(D_{K\sigma}\ln(\bu_i^n))^2\\
-\Delta t_n\!\!\sum_{K\in \Tt}\sum_{i=1}^N\sum_{\sigma\in \Ee_K}\!\!\tau_\sigma\left(\sum_{j=1}^N (a_{i,j}-a^\star)u_{j,\sigma}^nu_{i,\sigma}^n\ln(u_{i,K}^n)D_{K\sigma} \bigl(\ln(\bu_i^n)-\ln(\bu_j^n)\bigr)\right)=0.
\end{multline}
Using the symmetry of the matrix $A$ and discrete integration by part, both in space and with respect to the species, we have:
\begin{multline}\label{eq:crossdiffdiff}
\sum_{K\in \Tt}\sum_{i=1}^N\sum_{\sigma\in \Ee_K}\tau_\sigma\left(\sum_{j=1}^N (a_{i,j}-a^\star)u_{j,\sigma}^nu_{i,\sigma}^n\ln(u_{i,K}^n)D_{K\sigma} \bigl(\ln(\bu_i^n)-\ln(\bu_j^n)\bigr)\right)=\\
-\sum_{\sigma\in \Ee}\tau_\sigma\left(\sum_{1\leq i <j\leq N)} (a_{i,j}-a^\star)u_{j,\sigma}^nu_{i,\sigma}^n\Bigl(D_{K\sigma} \bigl(\ln(\bu_i^n)-\ln(\bu_j^n)\bigr)\Bigr)^2\right).
\end{multline}
On the other hand, the convexity of $c\ln(c)$ yields:
\[
u_{i,K}^n-u_{i,K}^{n-1}+u_{i,K}^n\ln(u_{i,K}^n)-u_{i,K}^{n-1}\ln(u_{i,K}^n) \geq u_{i,K}^n\ln(u_{i,K}^n)-u_{i,K}^{n-1}\ln(u_{i,K}^{n-1}).
\]
Combining this inequality with Equation \eqref{eq:crossdiffdiff} and Lemma~\ref{lem:conservation} in \eqref{eq:primalEED} provides:
\begin{multline*}
	E_\Tt(\bU^n)- E_\Tt(\bU^{n-1}) 
	+\Delta t_n a^\star\sum_{\sigma\in \Ee}\sum_{i=1}^N \tau_\sigma u_{i,\sigma}^n(D_{K\sigma}\ln(\bu_i^n))^2\\
	+\Delta t_n \sum_{\sigma\in \Ee}\tau_\sigma\left(\sum_{1\leq i <j\leq N} (a_{i,j}-a^\star)u_{j,\sigma}^nu_{i,\sigma}^n\Bigl(D_{K\sigma} \bigl(\ln(\bu_i^n)-\ln(\bu_j^n)\bigr)\Bigr)^2\right)\leq 0.
\end{multline*}
	Using the hypothesis $0\leq\min a_{i,j}\leq a^\star$ together with
	\begin{multline}\label{eq:astarsurdiff}
	\sum_{i=1}^N u_{i,\sigma}^n(D_{K\sigma}\ln(\bu_i^n))^2
	-\left(\sum_{1\leq i <j\leq N} u_{j,\sigma}^nu_{i,\sigma}^n\Bigl(D_{K\sigma} \bigl(\ln(\bu_i^n)-\ln(\bu_j^n)\bigr)\Bigr)^2\right) =\\
	\sum_{i=1}^N  u_{i,\sigma}^n\left(1-\sum_{j=1}^N u_{j,\sigma}^n\right) (D_{K\sigma}\ln(\bu_i^n))^2\geq 0,
	\end{multline}
	we deduce that \eqref{eq:EEDdiscrete} holds.
\end{proof}

The proof of Theorem~\ref{thm:main1} is now complete.

\section{Convergence analysis}\label{sec:convergence}{}

The goal of this Section is to prove Theorem~\ref{thm:main2}, which states the 
convergence of the approximate solution towards a weak solution 
to the continuous problem in the sense of Definition~\ref{def:weaksol}
under the nondegeneracy condition \eqref{eq:nondegenerate}. 
We could extend this result on several other special cases 
including the one treated in \cite{Anita2019}.
We hint that the optimal assumption would be that the zeros of the diffusion matrix 
form a cluster-graph. However, we stick to the study of the non-degenerate case for the sake of simplicity.

We consider here a sequence $\left(\Tt_m, \Ee_m, \left(x_K\right)_{K\in\Tt_m}\right)_{m\geq 1}$ 
of admissible discretizations with $h_{\Tt_m}, h_{T,m}$ tending to $0$ as $m$ tends to $+\infty$, 
while the regularity $\zeta_{\Tt_m}$ remains uniformly bounded from below by a positive constant $\zeta^\star$. 
Theorem~\ref{thm:main1} provides the existence of a family of discrete solutions 
$\bU_m= \left(u_{i,K}^n\right)_{i\in \lent 1, N\rent, K\in\Tt_m, 1 \leq n \leq N_m}$.
To prove Theorem~\ref{thm:main2}, we first establish in Section~\ref{ssec:compact} some compactness properties on the family 
of piecewise constant approximate solutions $U_{\Tt_m,\bdt_m}$. Then we identify 
the limit as a weak solution in Section~\ref{ssec:identify}. 
In order to enlighten the notations, we remove the subscript $m$ as soon as it is not necessary for understanding.

\subsection{Reconstruction operators}\label{ssec:reconstruct}
To carry out the convergence analysis, we introduce some reconstruction operators 
following the methodology proposed in~\cite{kangourou_2018}. 
The operators $\pi_\Tt: \R^\Tt \to L^\infty(\O)$ and $\pi_{\Tt,\bdt}:\left(\R^{\Tt}\right)^{N_T}\to L^\infty(Q_T)$ are defined respectively by 
\[
\pi_\Tt\bF(x) = f_K \quad\text{if}\; x \in K, \qquad \forall \bF = \left(f_K\right)_{K\in\Tt}, 
\]
and
\[
\pi_{\Tt,\bdt}\bF(t,x) = f_K^n \quad\text{if}\; (t,x) \in (t_{n-1},t_n] \times K, \qquad \forall \bF = \left(f_K^n\right)_{K\in\Tt, 1 \leq n \leq N_T}. 
\]
These operators allow to pass from the discrete solution ${(\bU^n)}_{1 \leq n \leq N_T}$ to the approximate solution since 
\[
u_{i,\Tt,\bdt} = \pi_{\Tt,\bdt} \left(\bu_i^n\right)_n, \quad \forall i \in \lent 1,N\rent.
\]

In order to carry out the analysis, we further need to introduce approximate gradient reconstruction. 
For $\sig = K|L \in \Ee_{\rm int}$, we denote by $\Delta_\sig$ the diamond cell corresponding to $\sig$, which is the interior of the convex hull of $\{\sig, x_K, x_L\}$. For $\sig \in \Ee_{\rm ext}$, the diamond cell $\Delta_\sig$ 
is defined as the interior of the convex hull of  $\{\sig, x_K\}$. 
The approximate gradient $\grad_{\Tt}: \R^{\Tt} \to L^2(\O)^d$ we use in the analysis is merely weakly consistent (unless $d=1$) and takes its source in~\cite{CHLP03, EG03}.
It is piecewise constant on the diamond cells $\Delta_\sig$, and it is defined as follows:
\[
\grad_{\Tt} \bF(x) =  d \frac{D_{K\sig} \bF}{d_\sig} n_{K\sig} \quad \text{if}\; x \in \Delta_\sig, \qquad \forall \bF \in \R^{\Tt},
\]
where $n_{K\sig}$ is the outer-pointing normal of $K$ at $\sig$.
We also define $\grad_{\Tt,\bdt}: \R^{\Tt\times N_T} \to L^2(Q_T)^d$ by setting 
\[
\grad_{\Tt,\bdt}\bF(t,\cdot) = \grad_\Tt \bF^n \quad \text{if}\; t\in(t_{n-1},t_n], \qquad \forall \bF = \left(\bF^n\right)_{1\leq n\leq N_T} \in \R^{\Tt\times N_T}.
\]
It follows from the definition of the approximate gradient that 
\begin{equation}\label{eq:scalarprod}
\sum_{\sig \in \Ee}\tau_\sig D_{K\sig} \bF D_{K\sig} \bg = \frac1d \int_\O \grad_{\Tt} \bF\cdot \grad_{\Tt} \bg \d x, 
\qquad \forall \bF, \bg \in \R^{\Tt}.
\end{equation}
This implies in particular that
\begin{equation}\label{eq:norm_L2H1}
\sum_{\sig \in \Ee}\tau_\sig |D_\sig \bF|^2 = \frac1d\int_{\O} |\grad_{\Tt} \bF|^2 \d x, \qquad \forall \bF \in  \R^{\Tt}.
\end{equation}

\subsection{Compactness properties}\label{ssec:compact}
 In this subsection, we take advantage of Proposition~\ref{prop:EEDdiscrete} 
 and of the non-degeneracy assumption~\eqref{eq:nondegenerate} to get enough compactness for the convergence.
\begin{lemma}\label{lem:L2H1_sqrt}
	There exists $C$ depending only on $\Omega$ and $\min_{i\neq j} a_{i,j}$ such that
	\[
		\sum_{i=1}^N\iint_{Q_T} |\grad_{\Tt_m, \bdt_m} \sqrt{\bu_{i,m}}|^2+ \left(\pi_{\Tt_m, \bdt_m} \sqrt{\bu_{i,m}}\right)^2\d x\d t \leq C, 
		\quad \forall m \geq 1.
	\]
\end{lemma}
\begin{proof}
We get rid of the subscript $m$ for the ease of reading. The $L^\infty$ bound on $\bU$ yields immediately the $L^2$ estimate on $\pi_{\Tt, \bdt}\sqrt{\bu_i}$. The proof thus consists in proving the bound on the discrete gradient. 
Le us focus on the proof of
\(
\iint_{Q_T} |\grad_{\Tt, \bdt} \sqrt{\bu_i}|^2\d x\d t \leq C
\)
for some fixed $i\in\lent 1,N \rent$.
Thanks to \eqref{eq:norm_L2H1}, we have
$$
\begin{aligned}
\iint_{Q_T} |\grad_{\Tt, \bdt} \sqrt{\bu_i}|^2&=d\sum_{n=1}^{N_T} \Delta t_n \sum_{\sigma\in\Ee_{\rm int}} \tau_\sig \vert D_\sig\sqrt{\bu_i^n}\vert^2,\\
&=d\sum_{n=1}^{N_T} \Delta t_n \sum_{\sigma\in\Ee_{\rm int}} \tau_\sig \check{u}_{i\sig}^{n}
\vert D_\sig\ln(\bu_i^n)\vert^2,
\end{aligned}
$$
where $\check{u}_{i\sig}^n=4\frac{\left(D_\sig\sqrt{\bu_i^n}\right)^2}{\left(D_\sig\ln(\bu_i^n)\right)^2}$. 
It results from Cauchy-Schwarz inequality that 
\[
4 \left(\sqrt{a} - \sqrt{b}\right)^2 \leq (a-b) (\ln(a)-\ln(b)), \qquad \forall (a,b) \in (0,+\infty), 
\]
so that 
$\check{u}_{i\sig}^n\leq u_{i\sig}^n$.
Therefore, Proposition~\ref{prop:EEDdiscrete} provides:
$$
\min_{i\neq j} a_{i,j} \sum_{i=1}^N \iint_{Q_T} |\grad_{\Tt, \bdt} \sqrt{\bu_i}|^2
\leq  \frac{d}4\left( E_\Tt(\bU^0)-E_\Tt(\bU^{N_T})\right).
$$
As $E_\Tt$ is bounded between $-m_\Omega$ and $0$ and as, by hypothesis, $\min a_{i,j}>0$, we obtain the 
desired bound.
\end{proof}

The inequality $2 D_\sig \sqrt{\bu_i^n}\geq D_\sig \bu_i^n$ and Lemma~\ref{lem:L2H1_sqrt} yield the following discrete 
$L^2(0,T;H^1(\O))$ estimate on $\bu_i$. 
\begin{corollary}\label{coro:lavraiecompacite}
There exists $C$ depending only on $\Omega$ and $\min_{i\neq j} a_{i,j}$ such that
\[
		\sum_{i=1}^N\iint_{Q_T} |\grad_{\Tt_m, \bdt_m} \bu_{i,m}|^2+ \left(\pi_{\Tt_m, \bdt_m} \bu_{i,m}\right)^2\d x\d t \leq C, 
		\qquad \forall m\geq1.
\]
\end{corollary}

The following proposition is about the relative compactness of the approximate solution and of the 
weakly consistent approximate gradient. 

\begin{prop}\label{prop:compact}
Let $(\bU_m)$ be the family of discrete solutions.
There exists at least one $U \in L^\infty(Q_T; \Aa)\cap L^2((0,T);H^1(\O))$ such that, up to a subsequence, for all $i\in\lent1,N\rent$:
\begin{align}\label{eq:conv_L2}
\pi_{\Tt_m, \bdt_m} \bu_{i,m}  &\underset{m\to\infty} \longrightarrow u_i \quad \text{strongly in $L^2(Q_T)$},
\\\label{eq:conv_grad_r}
\grad_{\Tt_m, \bdt_m} \bu_{i,m} &\underset{m\to\infty} \longrightarrow \grad u_i \quad \text{weakly in $L^2(Q_T)^d$}.
\end{align}
\end{prop}
\begin{proof}
	We  drop the subscript $m$ for clarity.
The proof of this result relies on a discrete Aubin-Lions lemma \cite[Lemma 3.4]{GL12} on the particular setting of \cite[Lemma 9]{Anita2019}. 
Define the discrete $L^2(0,T;(H^1(\O))')$ norm by duality as follows:
\[
\|\bv\|_{-1}=\sup\left\{ \int_\Omega \pi_{\Tt} \bv \pi_{\Tt} \bvarphi,\; \|\pi_{\Tt}\bvarphi\|_{L^2}^2+ \|\nabla_{\Tt}\bvarphi\|_{L^2}^2=1\right\}, 
\qquad \forall \bv \in \R^\Tt. 
\]
Therefore if $\|\nabla_{\Tt, \bdt} \bu_i\|_{L^2(Q_T)}\leq C$ and $\sum_n \|\bu_i^n-\bu_i^{n-1}\|_{-1}\leq C$, then, up to a subsequence, 
$\pi_{\Tt, \bdt} \bu_{i}$ tends towards some $u_i$ in $L^2(Q_T)$, while $\grad_{\Tt, \bdt} \bu_{i}$ 
converges weakly towards $\grad u_i$. In particular, $U \in L^2(0,T;H^1(\O))^N$.

Corollary \ref{coro:lavraiecompacite} provides the $L^2$ bound on $\nabla_{\Tt, \bdt} \bu_i$. 
For the other inequality, we let $\bvarphi \in \RR^{\Tt}$, $n\in\lent1,N_{T}\rent$ and $i\in\lent1,N\rent$. 
It follows from \eqref{eq:scheme_c} that
\[
\int_\Omega \pi_{\Tt}\left(\bu^n_{i}-\bu^{n-1}_{i}\right)\pi_{\Tt}\bvarphi = -\Delta t_n\sum_{K\in\Tt}\varphi_K \sum_{\sig\in\Ee_K}F_{i,K\sigma}^n.
\]
Using \eqref{eq:flux}, this yields
\begin{multline*}
\frac1{\Delta t_n}\int_\Omega \pi_{\Tt}\left(\bu^n_{i}-\bu^{n-1}_{i}\right)\pi_{\Tt}\bvarphi 
=\sum_{\sigma\in\Ee}a^\star\tau_\sig D_{K\sig}\bu_i^nD_{K\sig}\bvarphi\\
+\sum_{\sigma\in\Ee}\tau_\sigma\left(\sum_{j=1}^N (a_{i,j}-a^\star)\left(u_{j,\sigma}^nD_{K\sigma} u_i^n-u_{i,\sigma}^nD_{K\sigma} u_j^n\right)\right)D_{K\sig}\bvarphi.
\end{multline*}
Using the Cauchy-Schwarz inequality, the $L^\infty$ bound on $\left(u_{i,\sig}^n\right)_{\sig \in \Ee, i \in \lent 1,N \rent}$ 
and \eqref{eq:scalarprod} 
then leads to
\begin{multline*}
\frac1{\Delta t_n}\int_\Omega \pi_{\Tt}\left(\bu^n_{i}-\bu^{n-1}_{i}\right)\pi_{\Tt}\bvarphi 
\leq a^\star \|\nabla_{\Tt} \bu_i^n \|_{L^2(\Omega)}\|\nabla_{\Tt} \bvarphi \|_{L^2(\Omega)}\\
+\|\nabla_{\Tt} \bvarphi \|_{L^2(\Omega)}\sum_{j=1}^N |a_{i,j}-a^\star|\left(\|\nabla_{\Tt} \bu_i^n \|_{L^2(\Omega)}+\|\nabla_{\Tt} \bu_j^n \|_{L^2(\Omega)}\right).
\end{multline*}
By definition of the discrete $(H^1(\O))'$ norm, we have 
\[
\left\|\frac{\bu_i^n-\bu_i^{n-1}}{\Delta t_n}\right\|_{-1}\leq a^\star \|\nabla_{\Tt} \bu_i^n \|_{L^2(\Omega)}
+\sum_{j=1}^N |a_{i,j}-a^\star|\left(\|\nabla_{\Tt} \bu_i^n \|_{L^2(\Omega)}+\|\nabla_{\Tt} \bu_j^n \|_{L^2(\Omega)}\right).
\]
Using Corollary \ref{coro:lavraiecompacite} again provides that  $\sum_n \|\bu_i^n-\bu_i^{n-1}\|_{-1}\leq C$. 
The relative compactness properties on $\pi_{\Tt,\bdt} \bu_i$ and $\grad_{\Tt,\bdt} \bu_i$ follow.

We still have to prove that $U$ is in $L^\infty(Q_T; \Aa)$. Let $i \in \lent 1,N\rent$ and 
let $\varphi_i \in L^2(Q_T)$ be zero where the limit 
$u_i$ is nonnegative and 1 where the limit is negative, then
\[
\int_{Q_T} \varphi_i\pi_{\Tt, \bdt} \bu_{i}  \geq 0 \qquad \text{and}\qquad 
\int_{Q_T} \varphi_i\pi_{\Tt, \bdt} \bu_{i} \underset{m\to+\infty}\longrightarrow \int_{Q_T}u_i\varphi_i \leq 0.
\]
Therefore, $\int_{Q_T}u_i\varphi_i=0$, so that $u_i$ is nonnegative. Finally, 
the linearity of the limit yields $\sum_{i=1}^N u_i=1$.
\end{proof}
\begin{remark} The uniform $L^\infty(Q_T)$ bound on $\pi_{\Tt_m, \bdt_m} \bU_m$ together with the strong convergence 
in $L^2(Q_T)$ yield \eqref{eq:conv_Lp} thanks to H\"older's inequality:
\[
\pi_{\Tt_m, \bdt_m} \bU_m  \underset{m\to\infty} \longrightarrow U \quad \text{strongly in $L^p(Q_T)^N$, for any $1\leq p <\infty$}.\]
\end{remark}

We also need convergence properties for the face values $u_{i,\sigma}$. We can reconstruct an approximate solution $u_{i,\Ee,\bdt}$ which is piecewise constant on the diamond cells by setting, for all $i \in \lent 1, N\rent$:
\[
u_{i,\Ee,\bdt}(t, x) = u_{i,\sig}^{n}  \quad\text{if}\quad (t, x) \in (t_{n-1},t_n] \times \Delta_\sig, \quad \sig \in \Ee.
\]

\begin{lemma}\label{lem:c_Ee} We have, for any $i\in \lent 1, N \rent$:
\[
u_{i,\Ee_m,\bdt_m} \underset{m\to\infty}\longrightarrow u_i \quad \text{in}\quad L^p(Q_T)\text{, for any $1\leq p <\infty$,}
\]
where $U$ is as in Proposition~\ref{prop:compact}.
\end{lemma}
\begin{proof} Here again, we get rid of $m$ for clarity, and show the convergence for a specific value of $p$. 
The convergence for any finite $p$ follows from the $L^\infty(Q_T)$ bound on $u_{i,\Ee_m,\bdt_m}$ and H\"older's inequality.
Since $u_{i,\Tt,\bdt}$ converges towards $u_i$ in $L^1(Q_T)$, and since $u_{i,\Ee,\bdt}$ is uniformly bounded, 
it suffices to show that $\|u_{i,\Ee,\bdt}-u_{i,\Tt,\bdt}\|_{L^1(Q_T)}$ tends to $0$. 
Denote by $\Delta_{K\sig}$ the half-diamond cell which is defined as the interior of the convex hull of $\{x_K,\sig\}$ for $K\in\Tt$ 
and $\sig \in \Ee_K$, then the following geometrical relation holds:
 \[m(\Delta_{K\sig}) = \frac1d m_\sig {\rm dist}(x_K,\sig) \leq \frac{h_\Tt}d m_\sig.\]
 As a consequence, 
\begin{align*}
\|u_{i,\Ee,\bdt}-u_{i,\Tt,\bdt}\|_{L^1(Q_T)} = & \sum_{n=1}^{N_T} \dt_n \sum_{K\in\Tt} \sum_{\sig \in \Ee_K} m_{\Delta_{K\sig}}|u_{i,K}^n - u_{i,\sig}^n| \\
\leq & \frac{h_\Tt}{d}  \sum_{n=1}^{N_T} \dt_n \sum_{K\in\Tt} \sum_{\sig \in \Ee_K} m_\sig|u_{i,K}^n - u_{i,\sig}^n|.
\end{align*}
As we have $u_{i,K}^n = u_{i,\sig}^n$, the contributions corresponding to the boundary edges vanish.
For $\sig \in \Ee_{\rm int}$, $u_{i,\sig}$ is an average of $u_{i,K}$ and $u_{i,K\sig}$, hence
\(|u_{i,K}^n - u_{i,\sig}^n| \leq |u_{i,K}^n - u_{i,K\sig}^n|.\)
Therefore, we obtain that 
\begin{multline*}
\|u_{i,\Ee,\bdt}-u_{i,\Tt,\bdt}\|_{L^1(Q_T)} \leq \frac{h_\Tt}d\sum_{n=1}^{N_T} \dt_n \sum_{\sig \in \Ee_{\rm int}} 2 m_\sig \left|D_\sig \bu_i^n\right|\\
\leq2\frac{h_\Tt}d\left(\sum_{n=1}^{N_T} \dt_n \sum_{\sig \in \Ee_{\rm int}} m_\sig d_\sig\right)^\frac12\left(\sum_{n=1}^{N_T} \dt_n \sum_{\sig \in \Ee_{\rm int}} \tau_\sig \left|D_\sig \bu_i^n\right|^2\right)^\frac12.
\end{multline*}
We deduce from Corollary~\ref{coro:lavraiecompacite} that 
\(
\|u_{i,\Ee,\bdt}-u_{i,\Tt,\bdt}\|_{L^1(Q_T)} \leq C h_\Tt,
\)
hence $u_{i,\Ee,\bdt}$ and $u_{i,\Tt,\bdt}$ share the same limit in $L^1(Q_T)$.
\end{proof}

\subsection{Convergence towards a weak solution}\label{ssec:identify}

The last step to conclude the proof of Theorem~\ref{thm:main2} is to identify the limit 
value $U$ exhibited in Proposition~\ref{prop:compact} as a weak solution to \eqref{eq:crossdiff}, \eqref{eq:no-flux}  
corresponding to the initial profile $U \in L^\infty(\Omega; \Aa)$. This is the purpose of our last statement. 
\begin{prop}\label{prop:conv_c}
Let $U$ be as in Proposition~\ref{prop:compact}, then $U$ is a weak solution in the sense of Definition~\ref{def:weaksol}.
\end{prop}
\begin{proof}
We drop again the subscript $m$ for the sake of readability, and let $i\in\lent 1, N\rent$, $\varphi \in C^\infty_c([0,T)\times\ov\O)$, then define $\bvarphi=(\varphi_K^n)$ by $\varphi_K^n = \varphi(x_K,t_n)$ for all $n \in \{0,\dots, N_T\}$ and $K \in \Tt$.  Multiplying~\eqref{eq:scheme_c}  by $\dt_n\varphi_K^{n-1}$, then summing over $K\in\Tt$ and $n\in\{1,\dots, N_T\}$ leads to 
\begin{equation}\label{eq:T123}
T_1 + T_2 + T_3 = 0, 
\end{equation}
where we have set 
\begin{align*}
T_1 = &\sum_{n=1}^{N_T}\sum_{K\in\Tt}m_K(u_{i,K}^n - u_{i,K}^{n-1}) \varphi_K^{n-1},\\
T_2 = &\sum_{n=1}^{N_T}\dt_n\sum_{\sig \in \Ee}\tau_\sig a^\star  D_{K\sig}\bu^n_i D_{K\sig}\bvarphi^{n-1},\\
T_3 = &\sum_{n=1}^{N_T}\dt_n\sum_{\sig \in \Ee}\tau_\sig \sum_{j=1}^N (a_{i,j}- a^\star) \left(u_{j\sig}^n D_{K\sig}\bu_i^n-u_{i\sig}^n D_{K\sig}\bu_j^n\right) D_{K\sig}\bvarphi^{n-1}.
\end{align*}
The term $T_1$ can be rewritten as 
\[
T_1 = \sum_{n=1}^{N_T}\dt_n\sum_{K\in\Tt}m_K u_{i,K}^n \frac{\varphi_{K}^{n-1}-\varphi_K^n}{\dt_n} - \sum_{K\in\Tt}m_K u_{i,K}^0 \varphi_{K}^0,
\]
so that it follows from the convergence of $\pi_{\Tt,\bdt} \bU$ towards $U$ and of $\pi_\Tt \bU^0$ towards $U^0$ together with the regularity of $\varphi$ that 
\begin{equation}\label{eq:T1}
T_1 \underset{m\to \infty} \longrightarrow -\iint_{Q_T} u_i \p_t \varphi \d x\d t - \int_\O u_i^0 \varphi(0,\cdot)\d x. 
\end{equation}

To treat the term $T_2$, we introduce a strongly consistent reconstruction of the gradient. 
Following~\cite{DE_FVCA8} (see~\cite{CVV99} for a practical example), 
one can reconstruct a second approximate gradient operator $\wh{\grad}_{\Tt}: \R^\Tt \to L^\infty(\O)^d$ such that 
\[
\int_{\Delta_\sigma}\grad_\Tt \bu \cdot \wh\grad_\Tt \bv \d x = \tau_\sig D_{K\sig} \bu D_{K\sig} \bv, \qquad \forall \bu, \bv \in \R^\Tt, \forall \sigma\in\Ee,
\]
and which is strongly consistent, i.e., 
\[
\wh \grad_\Tt \bvarphi^n \underset{h_\Tt\to0}\longrightarrow \grad \varphi(\cdot, t_n)\;\;\text{uniformly in}\; \ov\O, \quad \forall n \in \{1,\dots, N_T\},
\]
thanks to the smoothness of $\varphi$. Using this tool, the terms $T_2$ and $T_3$, are easy to treat. The first one can be rewritten as:
\[
	T_2=a^\star\iint_{Q_T} \grad_{\Tt,\bdt} \bu_i \cdot \widehat{\grad}_{\Tt,\bdt} \bvarphi \d x\d t, 
\]
so that 
\begin{equation}\label{eq:T2}
T_2 \underset{m\to \infty}\longrightarrow a^\star\iint_{Q_T} \grad u_i \cdot \grad \varphi \d x\d t. 
\end{equation}
On the other hand, the term $T_3$ rewrites
\[
T_3=\iint_{Q_T}  \sum_{j=1}^N (a_{i,j}- a^\star) \left(u_{j,\Ee,\bdt} \grad_{\Tt,\bdt} \bu_i-u_{i,\Ee,\bdt} \grad_{\Tt,\bdt} \bu_j\right) \widehat{\grad}_{\Tt,\bdt} \bvarphi,
\]
so that
\begin{equation}\label{eq:T3}
T_3 \underset{m\to \infty} \longrightarrow \iint_{Q_T}  \sum_{j=1}^N (a_{i,j}- a^\star) \left(u_j \grad u_i-u_i \grad u_j\right) \grad\varphi.
\end{equation}
Combining \eqref{eq:T123}, \eqref{eq:T1}, \eqref{eq:T2}, and \eqref{eq:T3}, we obtain that
\begin{multline*}
-\iint_{Q_T} u_i \p_t \varphi \d x\d t - \int_\O u_i^0 \varphi(0,\cdot)\d x+ a^\star \iint_{Q_T} \grad u_i \cdot \grad \varphi \d x\d t\\+ \iint_{Q_T}  \sum_{j=1}^N (a_{i,j}- a^\star) \left(u_j \grad u_i-u_i \grad u_j\right) \grad\varphi\d x\d t=0, 
\quad 
\forall \varphi \in C^\infty_c([0,T)\times\ov\O).
\end{multline*}
Using $U\in\Aa$ and the relation \eqref{eq:changeofastar}, we recover the weak formulation~\eqref{eq:wf}.
\end{proof}
\section{Numerical results}\label{sec:num}
The numerical scheme has been implemented using MATLAB. The nonlinear system corresponding to the scheme is solved thanks to a variation of the Newton method with stopping criterion $\|\bU^{n,k+1}-\bU^{n,k}\|_{\infty}\leq 10^{-12}$ {for successful convergence or $20$ iteration for failed convergence}. The solution of the Newton iteration, $\bU^{n,k+1/3}$, is then ``projected'' on $\Aa$ by setting $\bU^{n,k+2/3}=\max(\bU^{n,k+1/3},10^{-10}\tau)$, and then for all $K\in\Tt$: $U_K^{n,k+1}=U_K^{n,k+2/3}/(\sum_{i=1}^N u_{i,K}^{n,k+2/3})$. {The seed of the algorithm is the solution to $N$ uncoupled heat equations with diffusion coefficients all equal to $a^\star$}.

For the first time step, we also make use of a continuation method based on the intermediate diffusion coefficients 
$a_{i,j}^\lambda=\lambda a_{ij}+(1-\lambda) a^\star$ with $\lambda \in [0,1]$. The parameter $\lambda$ is originally set to $1$. If Newton's method does not converge, we let $\lambda=(\lambda+\lambda_{\text{prev}})/{2}$ where $\lambda_{\text{prev}}$ is originally set to $0$. If Newton's method converges, we let $\lambda_{\text{prev}}=\lambda$ and $\lambda=1$. 

\subsection{Convergence under grid refinement}\label{ssec:1Dconv}

Our first test case is devoted to the convergence analysis of the scheme in a one-dimensional setting $\Omega=(0,1)$.
Two different initial conditions are considered: $U^0_s$ is smooth and vanished point-wise at the boundary of $\Omega$, whereas $U^0_r$ is discontinuous
and vanishes on intervals of $\O$:
\begin{align*}
&u_{1,s}^0(x)=\frac14+\frac14\cos(\pi x),&&
u_{2,s}^0(x)=\frac14+\frac14\cos(\pi x),&&
u_{3,s}^0(x)=\frac12-\frac12\cos(\pi x),\\
&u_{1,r}^0=1_{[\frac38,\frac58]},&&
u_{2,r}^0=1_{(\frac18,\frac38)}+1_{(\frac58,\frac78)},&&
u_{3,r}^0=1_{[0,\frac18]}+1_{[\frac78,1]}.
\end{align*}
We also consider {three cross-diffusion coefficients matrices, a first one $A^{\rm Lap}$ corresponding to $3$ uncoupled heat equation,} a second one $A^{\rm reg}$ called regular with positive off-diagonal coefficients, and a third one $A^{\rm sing}$ called singular with a few null off-diagonal coefficients: 
\[
{A^{\rm Lap}=\left(\begin{matrix}
0&  1& 1\\
1&  0&  1\\
1& 1&  0
\end{matrix}\right),
\qquad}
A^{\rm reg}=\left(\begin{matrix}
0&  0.2& 1\\
0.2&  0&  0.1\\
1& 0.1&  0
\end{matrix}\right),
\qquad
A^{\rm sing}=\left(\begin{matrix}
0&  0& 1\\
0&  0&  0.1\\
1& 0.1&  0
\end{matrix}\right).
\]

For the convergence tests, we have let $ a^\star=0.1$ and the meshes are uniform discretisations of $[0,1]$ from $2^5$ cells to $2^{15}$ cells. {The approximate solutions are compared to a reference solution which is analytical when $A=A^\text{Lap}$ and computed on the finest grid ($2^{15}$ cells) when $A =A^{\rm reg}$ or $A =A^{\rm sing}$}.
The final time is $0.25$, and the time discretisation is fixed with a time step of {$\dt=2^{-18}$. In the case $A = A^{\rm Lap}$, we also 
report in Figure~\ref{tab:conv} results with $128$ times finer time step, i.e. $\dt=2^{-25}$}. 
One notices that our scheme is second-order accurate in space in the setting presented in this paper ($A=A^{\rm reg}$), but only first-order accurate when confronted to non-diffusive discontinuities. We call non-diffusive discontinuities a spatial discontinuity of $u_1^0$ and $u_2^0$ (recall that $a_{1,2} = 0$ in $A^{\rm sing}$) for which $u_3^0$ is equal to $0$ on both sides of the discontinuity, so that the 
contributions corresponding to $a_{1,3}$ and $a_{2,3}$ vanish at $t=0$.
The origin of this lower order {might} lie in the difficulty to compute accurately the near-zero concentrations in the neighborhood of such discontinuities. {We also notice in Figure~\ref{tab:conv} the prevalence of the error in time when comparing with respect to an analytical solution, which will motivate the development of higher-order in time methods already discussed in Remark~\ref{rmk:higher-order}.}

\begin{figure}[htb]
	\begin{tikzpicture}
	\begin{loglogaxis}[
	xlabel=number of cells,
	ylabel=error in $L^2$ norm,
	legend style={
		cells={anchor=east},
		legend pos=outer north east,
	},
	width=0.6\linewidth]%the width is tailored to fit in the margins, please think twice before changing it.
	
	\addplot[mark=square*, color=blue] table[x=N, y=D1C0S] {conv.dat};
	\addlegendentry{$A=A^{\rm reg}$ and $U^0=U_s^0$ }
	\addplot[mark=*, color=red] table[x=N, y=D1C0RR] {conv.dat};
	\addlegendentry{$A=A^{\rm reg}$ and $U^0=U^0_r$}
	\addplot[mark=triangle*,mark size=2.8pt, color= magenta] table[x=N, y=D0C0S] {conv.dat};
	\addlegendentry{$A=A^{\rm sing}$ and $U^0=U^0_s$}
	\addplot table[x=N, y=D0C0RR] {conv.dat};
	\addlegendentry{$A=A^{\rm sing}$ and $U^0=U^0_r$}
	\addplot table[x=N, y=DlapC0S] {conv.dat};
	\addlegendentry{$A=A^{\rm Lap}$ and $U^0=U^0_s$}
	\addplot table[x=N, y=DlapC0STS] {conv.dat};
	\addlegendentry{idem with $\dt_n=2^{-25}$}
	\logLogSlopeTriangle{0.1}{-0.4}{0.1}{1}{black};
	\logLogSlopeTriangle{0.1}{-0.4}{0.1}{2}{black};
	\end{loglogaxis}
\end{tikzpicture}
	\caption{Error with respect to reference solution.}
	\label{tab:conv}
\end{figure}
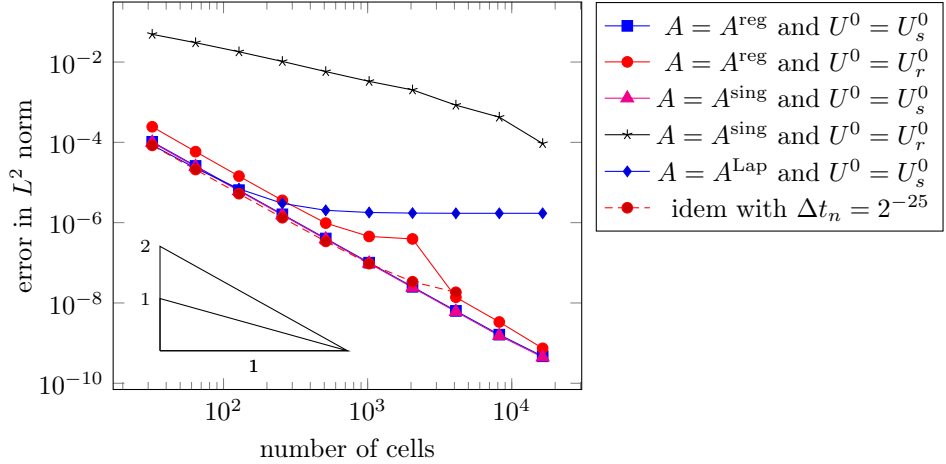

\subsection{On the influence of the parameter $a^\star$}
The choice of $a^\star$ is a natural question concerning our scheme. 
The equation \eqref{eq:choiceofa} gives a lower bound: $a^\star>0$. 
The existence of an upper bound is not as clear. Equation \eqref{eq:astarsurdiff} 
shows that for large $a^\star$, we over-estimate the diffusion. 
The optimal value of $a^\star$ depends on many variables such as the initial condition, the final time, and the mesh. 
Optimal choices of $a^\star$ are reported in Table \ref{tab:aopt}. Notice that the optimal value is test cases dependent,
since it is affected by the initial condition and by the final time. 
\begin{table}[htb]
	
\begin{tabular}{lll|c|c|c|c|}
\cline{4-7}
                                                    &                                           &           & \multicolumn{2}{c|}{$A=A^{\rm reg}$} & \multicolumn{2}{c|}{$A=A^{\rm sing}$} \\ \cline{4-7} 
                                                    &                                           &           & $U^0=U^0_s$       & $U^0=U^0_r$      & $U^0=U^0_s$       & $U^0=U^0_r$       \\ \hline
\multicolumn{1}{|l|}{\multirow{4}{*}{nb. of cells}} & \multicolumn{1}{l|}{\multirow{2}{*}{32}}  & $T=0.125$ & 0.86              & 0.21             & 0.79              & 0.0023            \\ \cline{3-7} 
\multicolumn{1}{|l|}{}                              & \multicolumn{1}{l|}{}                     & $T=0.25$  & 0.67              & 0.13             & 0.49              & 0.00082           \\ \cline{2-7} 
\multicolumn{1}{|l|}{}                              & \multicolumn{1}{l|}{\multirow{2}{*}{128}} & $T=0.125$ & 0.86              & 0.17             & 0.79              & 0.00050           \\ \cline{3-7} 
\multicolumn{1}{|l|}{}                              & \multicolumn{1}{l|}{}                     & $T=0.25$  & 0.67              & 0.11             & 0.49              & 0.00049           \\ \hline
\end{tabular}

\caption{Values of $a^\star_\text{opt}$ for different parameters. $a^\star_\text{opt}$ is computed with respect to the reference solution of Section \ref{ssec:1Dconv} for the $L^2$ norm.}
\label{tab:aopt}
\end{table}

One notices on Fig.~\ref{fig:astar} that the dependency of the quality of the results is {strong} for the initial data $U^0_r$.
This is due to the presence of vanishing concentrations in some cells, so that the choice $a^\star=0$ would allow 
for spurious solutions as highlighted in Remark~\ref{rmk:astar}. 
In this situation, the choice of $a^\star$ strongly affects the quality of the results, especially for the first time steps 
where some concentrations are still close to $0$.
The numerical experiment and homogeneity considerations suggest the following suboptimal rule for choosing $a^\star$:
\be\label{eq:rule.astar}
a^\star = \min\left\{ \max_{i\neq j} a_{i,j} \; ; \; \max\left\{ \min_{i\neq j} a_{i,j},\;  \epsilon \frac{h_\Tt^2}\tau \right\} \right\}, 
\ee
where $h_\Tt$ is the mesh size, $\tau$ the current time step and $\epsilon$ a small parameter to be tuned by the user. 
{Another interesting feature of Figure~\ref{fig:astar} is the behavior of the curve corresponding to $A=A^{\rm Lap}$. 
The classical TPFA scheme for the heat equation corresponding to $a^\star =1$ is outperformed in terms of 
accuracy by the scheme corresponding to higher value of $a^\star~\simeq 9.$ The introduction of enhanced diffusion 
by picking high values of $a^\star$ is not covered by formula~\eqref{eq:rule.astar}.
}

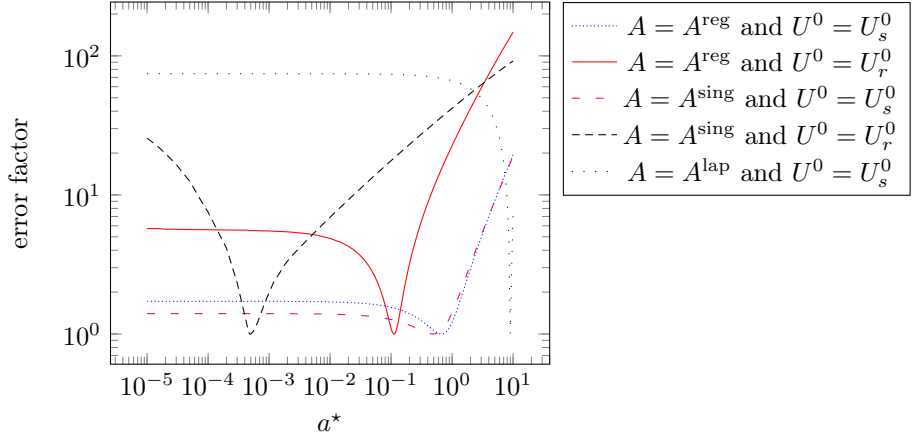
\begin{figure}[htb]
	\begin{tikzpicture}
	\begin{loglogaxis}[
	xlabel=$a^\star$,
	ylabel=error factor,
	legend style={
		%cells={anchor=east},
		legend pos=outer north east,
	},
	width=0.57\linewidth]%the width is tailored to fit in the margins, please think twice before changing it.
	
	\addplot[color=blue,densely dotted] table[x=A, y=D1C0S] {Errfactor.txt};
	\addlegendentry{$A=A^{\rm reg}$ and $U^0=U_s^0$ }
	\addplot[ color=red] table[x=A, y=D1C0RR] {Errfactor.txt};
	\addlegendentry{$A=A^{\rm reg}$ and $U^0=U^0_r$}
	\addplot[color= magenta,loosely dashed] table[x=A, y=D0C0S] {Errfactor.txt};
	\addlegendentry{$A=A^{\rm sing}$ and $U^0=U^0_s$}
	\addplot[color=black, densely dashed] table[x=A, y=D0C0RR] {Errfactor.txt};
	\addlegendentry{$A=A^{\rm sing}$ and $U^0=U^0_r$}
	\addplot[color=black, loosely dotted] table[x=A, y=DlapC0S] {Errfactor.txt};
	\addlegendentry{$A=A^{\rm lap}$ and $U^0=U^0_s$}
	
	\end{loglogaxis}
\end{tikzpicture}
	\caption{Evolution of the ratio $\frac{\|U_{a^\star}-U_{\rm ref}\|_2}{\|U_{a^\star_{\rm opt}}-U_{\rm ref}\|_2}$, where $U_{a^\star}$ is computed with $2^7$ cells and $U_{\rm ref}$ is as in Section \ref{ssec:1Dconv}.}
	\label{fig:astar}
\end{figure}

\subsection{A 2D test case {with reaction}}
Our second test is two-dimensional. We choose $A^{\rm sing}$ as the diffusion matrix and $ a^\star=0.1$.
The domain $\Omega=(0,22)\times(0,16)$ is discretized into a cartesian grid made of $110\times 80$ cells {which is presented along with initial condition in Figure \ref{fig:init}}. 
We use a uniform time stepping with $\tau=2^{-3}$. {To illustrate Remark~\ref{rmk:reac}, we introduce the following reaction:
	\[
	e_1+e_3\mathop{\rightleftharpoons}^{1000}_1 2e_2,
	\]
	which translates as follows in the {source term $R = \left(r_1,r_2,r_3\right)^T$:}
\[
	r_1(U)={(u_2^+)^2-1000u_1^+u_3^+}\qquad r_2(U)=-2r_1(U)\qquad r_3(U)=r_1(U).
\]
The reaction term $R$ obviously satisfies Assumptions \eqref{it:reac.iso} and \eqref{it:reac.pos} of Remark~\ref{rmk:reac}. 
{Let us now discuss Assumption~\eqref{it:reac.ent}.} We have for all $U,\ov{U}\in \Aa$:
\[
R(U).\ln(U/\ov{U})=r_1(U)\Bigl(\left[\ln(u_1u_3)-\ln(u_2^2)\right]-\left[\ln(\ov{u_1}\ov{u_3})-\ln(\ov{u_2}^2)\right]\Bigr).
\]
The above expression is then nonpositive for any $\ov{U}$ such that $\ln(\ov{u_1}\ov{u_3})-\ln(\ov{u_2}^2)=-\ln(1000)$}, 
{or equivalently $R(\ov U)  = 0$. A particular choice for such a $\ov U$ is the steady state $U^\infty$, which is constant 
w.r.t. to space and determined as follows. Denote by $\alpha$ the average advancement of the reaction, then}
{
\[
	u_1^\infty(x)=\frac{9}{44}-\alpha,\quad u_2^\infty(x)=\frac{2}{11}+2\alpha,\quad u_3^\infty(x)=\frac{27}{44}-\alpha.
\]
{F}or $\alpha=0$, the fraction of each specie is the ratio of corresponding occupied area in Figure~\ref{fig:init}{, i.e.
$\frac1{m_\O}\int_\O U^0=(\frac9{44},\frac2{11},\frac{27}{44})^T$}. 
{The value of $\alpha$} is determined by imposing that $R(U^\infty)= 0$, 
which amounts to find a root of a polynomial of degree two. 
{Among the two roots, only the choice $\alpha=\frac{-5\sqrt{206530}+4504}{10956}$
 yields a non-negative $U^\infty$.
 }
The time evolution of the relative energy $E_\Tt({\bU}\vert{\bU}^\infty)$ is plotted 
on Figure~\ref{fig:entropy}, showing exponential decay to the steady-state even though the diffusion matrix is singular.} Snapshots showing the evolution of the concentration profiles are presented in Figure~\ref{fig:snap}.

{To compute the solutions to our numerical scheme, we have to adapt the continuation procedure {sketched at the beginning of Section~\ref{sec:num}} to include source terms. {Roughly speaking}, we solve {the discrete counterpart to}
\[
	\partial_t u_i -a^\star\lap u_i -\lambda\dive\left(\sum_{j=1}^{N}(a_{i,j}-a^\star)\bigl(u_j\grad u_i- u_i\grad u_j\bigr)\right)=\mu r_i(U),
\]
{where the source terms are discretized in a fully implicit way.}
Due to the stiffness of the reaction terms, we have to treat the reaction first then the cross-diffusion effects.}
{More precisely, given $\bU^{n-1} \in \Aa^\Tt$, we initialise the iterative method for the computation of $\bU^n$ with 
$\bU^{n,0}$ defined as the unique solution to the $N$ uncoupled heat equations corresponding to $\lambda = \mu = 0$.
Then one tries to solve the system corresponding to $\lambda = \mu = 1$. If the modified Newton's method with truncation 
and rescaling sketched at the beginning of Section~\ref{sec:num} fails to converge, then one sets $\lambda = 0$ and $\mu = \frac12$. 
Then one use a similar continuation method to the one described at the beginning of Section~\ref{sec:num} to increase $\mu$ until 
it reaches the value $1$. Then the continuation method is used again to increase the value of $\lambda$ until $\lambda = 1$ is reached.}

\begin{figure}[htb]
		\centering
		\begin{tikzpicture} 
		\node at (0,0) {\includegraphics[width=6cm]{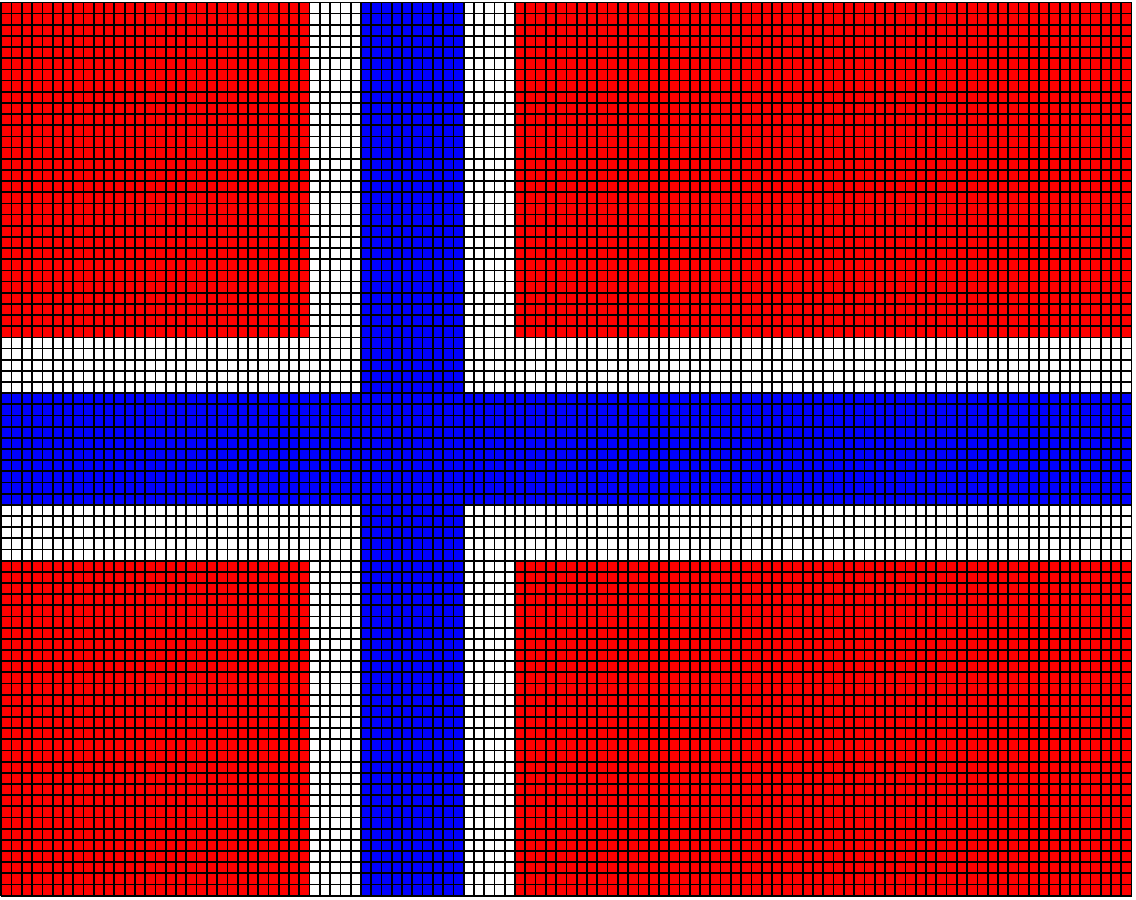}}; 
		\draw[black, fill=blue] (4.1,1)--+(.3,0)--+(.3,.2)--+(0,.2)--cycle;
		\node at (5,1.1) {$u_{1}^0=1$}; 
		\draw[black] (4.1,.5)--+(.3,0)--+(.3,.2)--+(0,.2)--cycle;
		\node at (5.,0.6) {$u_{2}^0=1$}; 
		\draw[black, fill=red] (4.1,0)--+(.3,0)--+(.3,.2)--+(0,.2)--cycle;
		\node at (5,.1) {$u_{3}^0=1$}; 
		\end{tikzpicture}
		\captionof{figure}{Initial configuration $U^0$ for the concentrations}  \label{fig:init}
		
		\centering
		\begin{tikzpicture}
		\begin{semilogyaxis}[xlabel=time,ylabel=relative entropy, width=0.5\linewidth]
		\addplot[mark=none] table[x=T, y=Erel] {entshort.dat};
		\end{semilogyaxis}
		\end{tikzpicture}
		\captionof{figure}{{$\left\vert E_\Tt({\bU}\vert{\bU}^\infty)\right\vert$} as a function of time.}
		\label{fig:entropy}
\end{figure}

\begin{figure}[htb]
	\centering
	\begin{subfigure}[t]{0.47\linewidth}
		\includegraphics[width=0.99\linewidth]{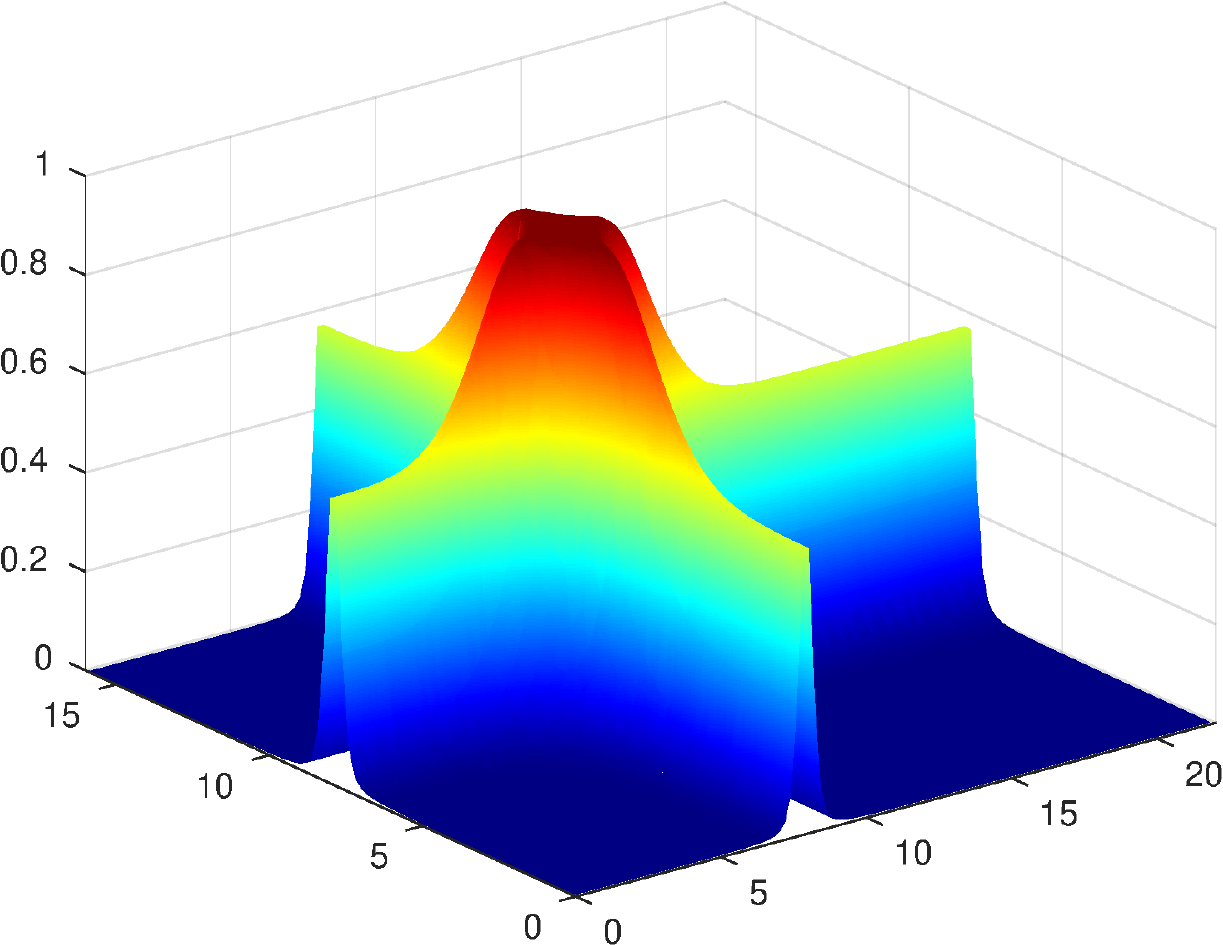}
		\caption*{Profile of $u_1$ at time {$t=20$}}
	\end{subfigure}
	\begin{subfigure}[t]{0.47\linewidth}
		\includegraphics[width=0.99\linewidth]{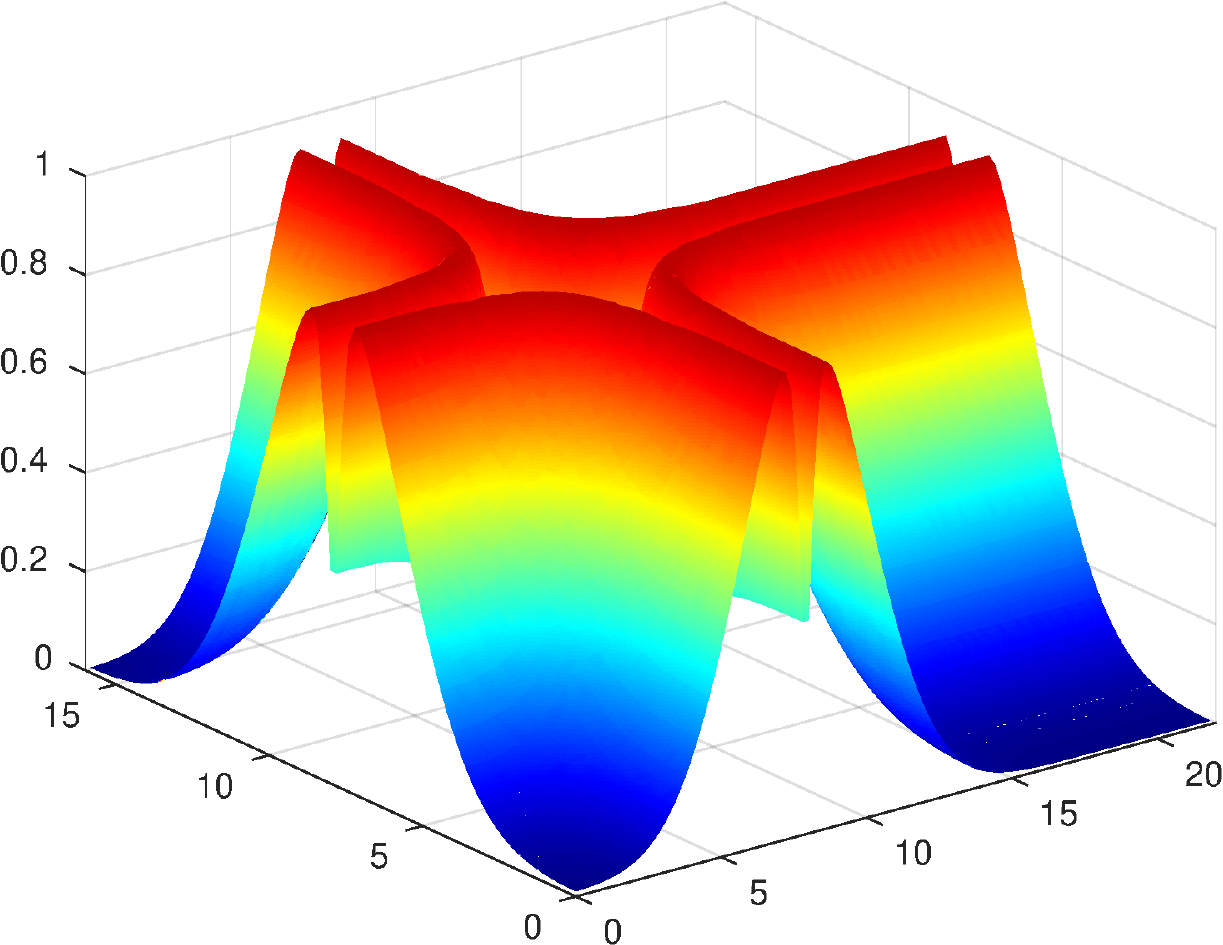}
		\caption*{Profile of $u_2$ at time {$t=20$}}
	\end{subfigure}
	\begin{subfigure}[t]{0.47\linewidth}
		\includegraphics[width=0.99\linewidth]{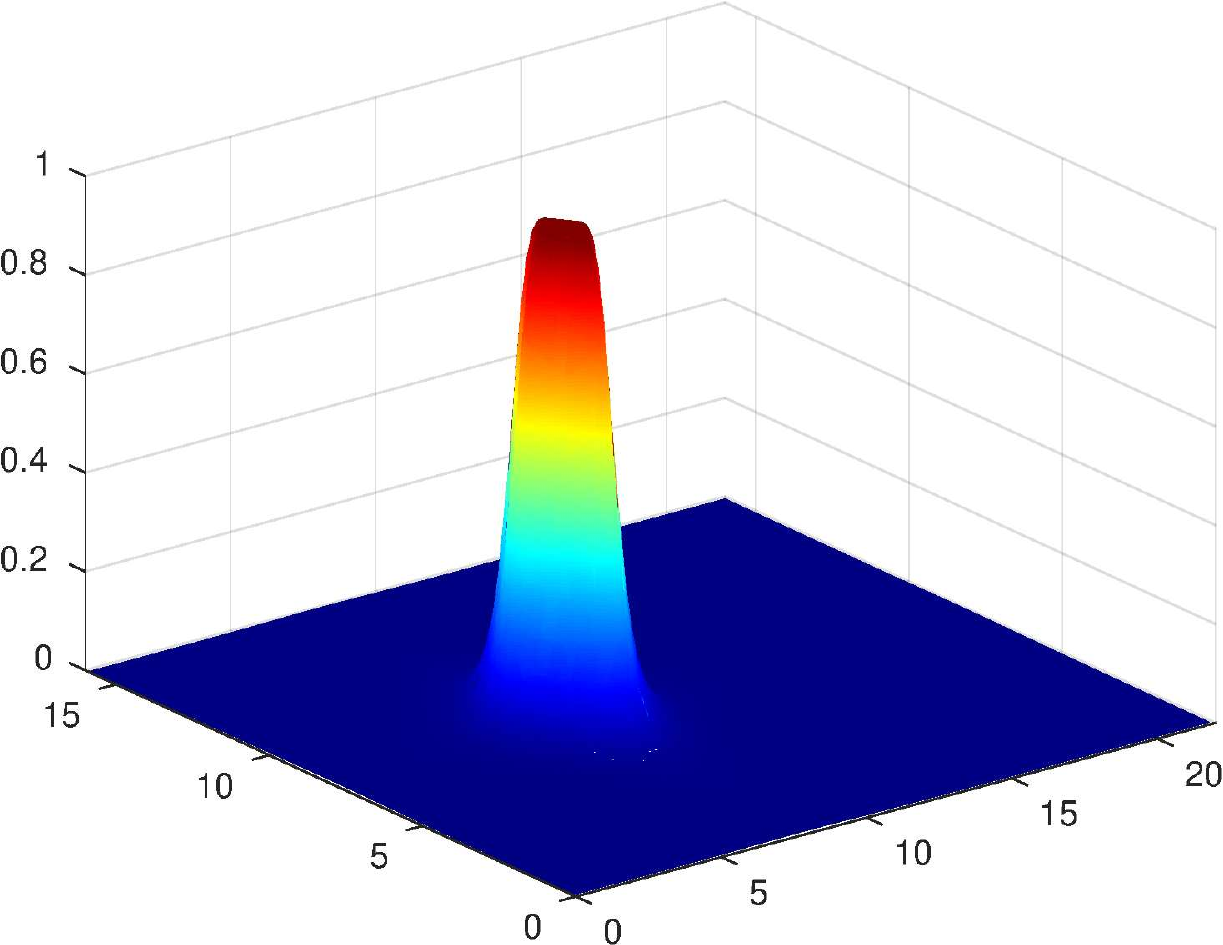}
		\caption*{Profile of $u_1$ at time {$t=50$}}
	\end{subfigure}
	\begin{subfigure}[t]{0.47\linewidth}
		\includegraphics[width=0.99\linewidth]{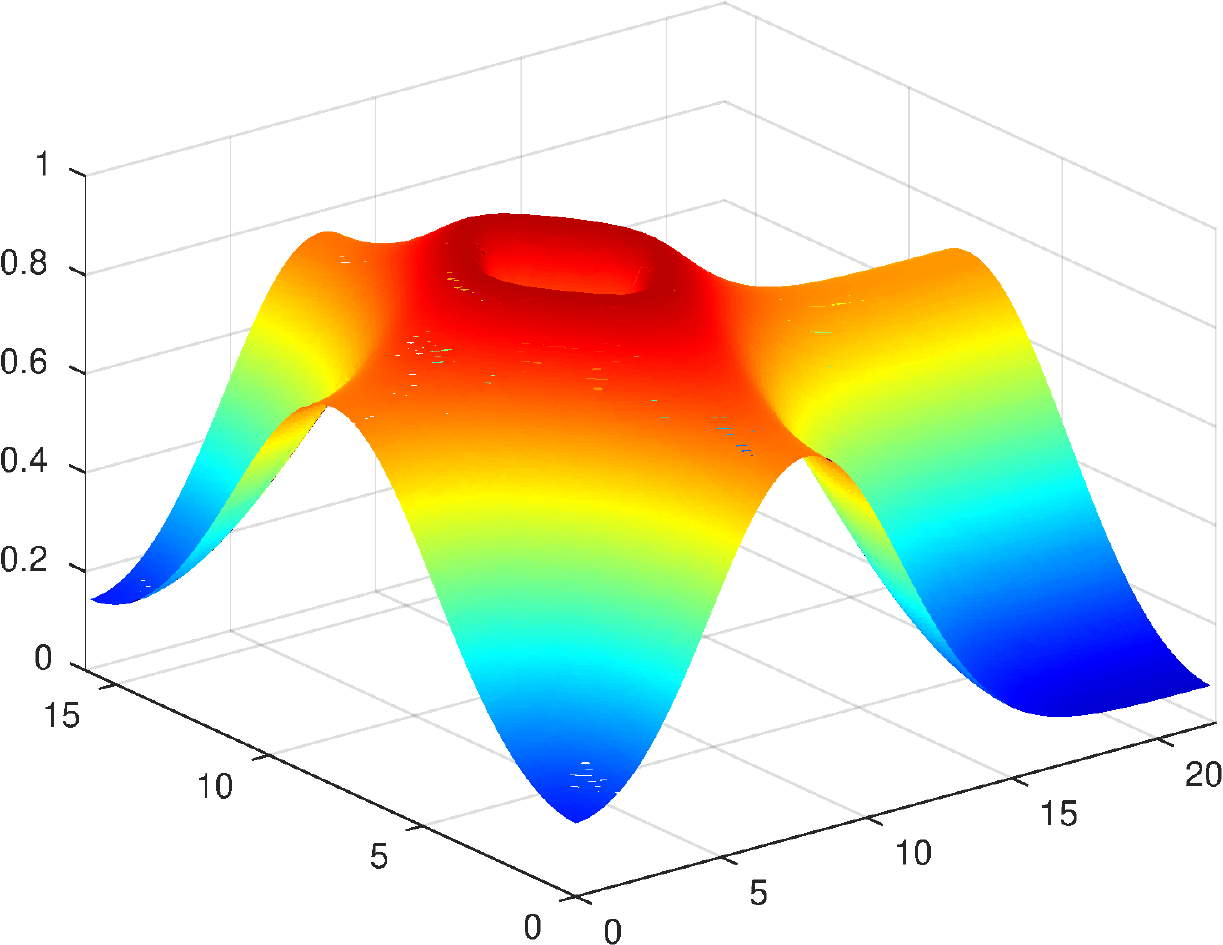}
		\caption*{Profile of $u_2$ at time {$t=50$}}
	\end{subfigure}
	\caption{Concentration configurations for various times. 
	The concentration of the third specie can be deduced thanks to $u_1+u_2+u_3=1$}
	\label{fig:snap}
\end{figure}

\section{Conclusion}
We proposed a finite volume scheme based on two-point flux approximation for 
a degenerate cross-diffusion system. The scheme was designed to preserve
the key properties of the continuous system, namely the positivity of the solutions, 
the constraint on the composition and the decay of the entropy. 
The scheme requires the introduction of a positive parameter $a^\star$ to avoid 
unphysical solutions. This parameter plays an important role in the convergence 
proof, which is carried out under a non-degeneracy assumption. Its importance 
is also confirmed in the numerical experiments, in particular in the presence of 
initial profiles with concentrations vanishing in some parts of the computational domain. 

\section*{Acknowledgements} 
The authors acknowledge support from the Labex CEMPI (ANR-11-LABX-0007-01). 
Cl\'ement Canc\`es also acknowledges support from the COMODO project (ANR-19-CE46-0002), 
and he warmly thanks Virginie Ehrlacher and Laurent Monasse for stimulating discussions that 
were at the origin of this work.

\end{document}